\documentclass[twoside, reqno]{amsart}

\usepackage{amsfonts,amsmath,amscd,amsthm,amssymb}
\usepackage[mathscr]{euscript} 
\usepackage{graphics}
\usepackage{wrapfig}
\usepackage{multirow,color}
\usepackage{lscape}
\usepackage{rotating}
\usepackage{cite}
\usepackage[vlined,ruled]{algorithm2e}

\newtheorem{lemma}{Lemma}
\newtheorem{theorem}{Theorem}

\renewcommand{\b}[1]{{\boldsymbol{#1}}}
\renewcommand{\c}[1]{{\mathcal{#1}}}

\newcommand{\RR}{\mathbb{R}}
\newcommand{\NN}{\mathbb{N}}
\newcommand{\ZZ}{\mathbb{Z}}

\DeclareMathOperator{\Grad}{\nabla}
\DeclareMathOperator{\Div}{\nabla\cdot}
\DeclareMathOperator{\Curl}{\nabla \times}

\providecommand{\RT}[1]{\pol_{#1}^3\oplus \b x \widetilde{\pol}_{#1}}

\providecommand{\NDOne}[1]{\pol_{#1}^3\oplus \b x\times \widetilde{\pol}_{#1}^3}
\providecommand{\NDTwo}[1]{\pol_{#1}^3}

\providecommand{\IND}[1]{\c{I}^{#1}_d}
\providecommand{\INTwo}[1]{\c{I}^{#1}_2}
\providecommand{\INThree}[1]{\c{I}^{#1}_3}

\providecommand{\BB}[2]{B^{#1}_{#2}}

\providecommand{\Xa}[2]{\b{\Phi}^{#1}_{#2}}
\providecommand{\Xb}[2]{\b{\Psi}^{#1}_{#2}}
\providecommand{\Xc}[2]{\b{\Upsilon}^{#1}_{#2}}

\newcommand{\trh}[1]{\mathrm{tr}_{#1}}
\newcommand{\trt}[1]{\b{\mathrm{{tr}}}_{t #1}}
\newcommand{\trn}[1]{\mathrm{tr}_{n#1}}
\newcommand{\NED}{N\'ed\'elec }

\def\restrict#1{\raise-0.2ex\hbox{\ensuremath|}_{#1}}

\newcommand{\cd}{c_{\b\alpha}}
\newcommand{\cdt}{\widetilde{c}_{\b\alpha}}
\newcommand{\cdn}{{c}_{\b\alpha+\b e_i-\b e_j}}

\newcommand{\ca}{c_{\b\alpha}^{(1)}}
\newcommand{\cb}{c_{\b\alpha}^{(2)}}
\newcommand{\cc}{c_{\b\alpha}^{(3)}}

\newcommand{\SSS}{S}
\newcommand{\EEE}{\b E}
\newcommand{\VVV}{\b V}
\newcommand{\WWW}{ W}

\begin{document}
\title[BB Basis for Tetrahedral Elements]{Bernstein-B\'ezier Bases for
Tetrahedral Finite Elements}
\author{Mark Ainsworth}
\address{Division of Applied Mathematics, Brown University, 182 George St,
Providence RI 02912, USA.}
\email{Mark\_Ainsworth@brown.edu}

\author{Guosheng Fu}
\address{Division of Applied Mathematics, Brown University, 182 George St,
Providence RI 02912, USA.}
\email{Guosheng\_Fu@brown.edu}


\newcommand{\pol}{\mathbb{P}}

\begin{abstract}
We present a new set of basis functions for 
$H(\mathrm{curl})$-conforming, $H(\mathrm{div})$-conforming, and 
$L^2$-conforming finite elements of arbitrary order
on a tetrahedron. The basis functions are expressed in terms of 
Bernstein polynomials and augment the natural $H^1$-conforming
Bernstein basis.
The basis functions respect the differential operators, namely, 
the gradients of the high-order $H^1$-conforming Bernstein-B\'ezier basis functions 
form part of
the $H(\mathrm{curl})$-conforming
basis, and the curl of the 
high-order, non-gradients $H(\mathrm{curl})$-conforming basis functions 
form part of
the $H(\mathrm{div})$-conforming basis, and the 
divergence of the 
high-order, non-curl $H(\mathrm{div})$-conforming basis functions 
form part of
the $L^2$-conforming basis.

Procedures are given for the
efficient computation of the mass and stiffness  matrices with 
these basis functions without using quadrature rules 
for (piecewise) constant coefficients on affine tetrahedra.
Numerical results are presented to 
illustrate the use of the basis to approximate representative problems.
\end{abstract}
\maketitle

\section{Introduction}
Let $\pol_{n}$ be the space of polynomials of degree no greater than $n$ on the tetrahedron, 
$\widetilde{\pol}_n$ is the space of homogeneous polynomials of degree $n$, and 
$\pol_{n}^3:= [\pol_{n}, \pol_{n}, \pol_{n}]^T$ is the vector space whose components lie in $\pol_n$.
We present a set of Bernstein-B\'ezier bases for the following 
two families of arbitrary order polynomial exact sequences 
on a tetrahedron $T$:
\begin{table}[!ht]
\caption{Two exact sequences on a tetrahedron $T$}
\begin{tabular}{l c c c c c c c}
\vspace{0.04cm}
&$H^1(T)
$&
&
$
H(\mathrm{curl};T)$
&
&
$
H(\mathrm{div};T)$
&
&
$
L^2(T)$
\\
\sf{First sequence:}
&$\pol_{n+1}
$&
$\overset{\Grad}{\longrightarrow}$ &
$
\pol_{n}^3
\oplus\b x
\times{\widetilde\pol}_{n}^3
$
&
$\overset{\Curl}{\longrightarrow}$ &
$
\pol_{n}^3
\oplus \b x\,\widetilde\pol_n
$&
$\overset{\Div}{\longrightarrow}$ &
$\pol_{n}
$\\
\sf{Second sequence:}
&$\pol_{n+2}
$&
$\overset{\Grad}{\longrightarrow}$ &
$
\pol_{n+1}^3
$
&
$\overset{\Curl}{\longrightarrow}$ &
$
\pol_{n}^3
$&
$\overset{\Div}{\longrightarrow}$ &
$\pol_{n-1}
$
\end{tabular}
\label{table-ex}
\end{table} 

The basis for the $H^1$-case is taken to be the standard Bernstein polynomials as in 
\cite{AinsworthAndriamaroDavydov11}. The basis for the 
$H(\mathrm{curl})$, $H(\mathrm{div})$ and $L^2$ cases are also constructed in terms of Bernstein polynomials
in such a way that:
the gradients of the high-order $H^1$-conforming Bernstein-B\'ezier basis functions are part of 
the $H(\mathrm{curl})$-conforming
basis; the curl of the 
high-order, non-gradient $H(\mathrm{curl})$-conforming basis functions are part of 
the $H(\mathrm{div})$-conforming basis; and, the 
divergence of the 
high-order, non-curl $H(\mathrm{div})$-conforming basis functions are part of 
the $L^2$-conforming basis.

The construction of high-order finite element basis functions for the polynomial exact sequences 
has been extensively studied in the literature, see the 
hierarchical bases in \cite{AinsworthCoyle03,Zaglmayr06,FuentesKeithDemkowiczNagaraj15},
and the Bernstein-B\'ezier bases in \cite{ArnoldFalkWinther09,AinsworthAndriamaroDavydov11,AinsworthAndriamaroDavydov15}.
Previous work \cite{AinsworthAndriamaroDavydov11} has demonstrated that using
the Bernstein polynomials as basis functions for the $H^1$-case offers a number of 
advantages including
enabling 
optimal complexity assembly  of the stiffness and mass matrices 
even for nonlinear problems on non-affine triangulations.
The current work shows that these, and other advantages of the Bernstein-B\'ezier basis are not confined
to the $H^1$-case but extend to the whole de Rham sequence.

Quite apart from computational consideration, the Bernstein polynomials provide a convenient tool for 
theoretical work on splines \cite{LaiSchumaker07} and were used in \cite{ArnoldFalkWinther09} to 
present alternative set of basis functions for the de Rham sequence.
Comparing with the Bernstein-B\'ezier basis introduced in \cite{ArnoldFalkWinther09}, our basis 
enjoys the so-called {\it local exact sequence} property \cite{Zaglmayr06}, which means that 
the exact sequence property holds at the basis function level: gradients of the high-order $H^1$-conforming 
basis functions form part of  the $H(\mathrm{curl})$-conforming basis functions;
curl of the high-order, non-gradient $H(\mathrm{curl})$-conforming 
basis functions form part of the $H(\mathrm{div})$-conforming basis functions; and, 
divergence of the high-order, non-curl $H(\mathrm{div})$-conforming 
basis functions form part of the $L^2$-conforming basis functions.
The clear separation of the gradients and non-gradients for the 
$H(\mathrm{curl})$-conforming basis, and divergence-free functions and non-divergence-free functions for the 
$H(\mathrm{div})$-conforming basis is not only satisfying from a theoretical standpoint, but can also be 
exploited in the practical application of the bases.

The rest of the paper is organized as follows: in Section \ref{sec:sequence}, we present the main results of the
basis functions for the two sequences in Table \ref{table-ex}.
In Section \ref{sec:curl} and Section \ref{sec:div}, we present details of the construction of the 
$H(\mathrm{curl})$-conforming basis, and the $H(\mathrm{div})$-conforming basis, respectively.
Then, in Section \ref{sec:comp}, we discuss the efficient implementation of our basis for constant coefficient problems on affine tetrahedral 
elements.
We conclude in Section \ref{sec:num} with numerical results validating our theoretical findings.

\section{Main results}
\label{sec:sequence}
In this section, we summarize the  main results concerning 
the Bernstein-B\'ezier basis functions for the 
$H(\mathrm{curl})$- and $H(\mathrm{div})$-conforming finite elements
on a tetrahedron $T := \mathrm{conv}(\b x_1,\b{x}_2, \b x_3, \b x_4)$.
Combined with the Bernstein-B\'ezier basis for $H^1$-conforming finite elements \cite{Ainsworth14}, 
we thereby obtain 
a set of basis functions, which 
respect the differential operators, 
for the two polynomial de Rham sequences in Table \ref{table-ex}, which
corresponding to \NED spaces of the first and second kinds.
We begin by collecting  some useful notation with  which to present the bases.

\subsection{Notation, index sets, and domain points}
\label{sec:notation}
Standard multi-index notations will be used throughout. In particular, for
$\b{\alpha}\in\ZZ_+^{d+1}$, we define $|\b{\alpha}|=\sum_{k=1}^{d+1}\alpha_k$,
$\b{\alpha}!=\prod_{k=1}^{d+1} \alpha_k!$ and
${|\b{\alpha}|\choose\b{\alpha}}=|\b{\alpha}|!/\b{\alpha}!$. If
$\b{\lambda} = (\lambda_1,\lambda_2,\cdots, \lambda_{d+1})\in\RR^{d+1}$, then we define 
$\b{\lambda}^\b{\alpha}=\prod_{k=1}^d
\lambda_k^{\alpha_k}$. Given a pair $\b{\alpha},\b{\beta}\in\ZZ_+^{d+1}$,
$\b{\alpha}\le\b{\beta}$ if and only if $\alpha_k\le\beta_k$, $k=1,\ldots,d+1$
and, in this case, ${\b{\beta}\choose\b{\alpha}}= \prod_{k=1}^{d+1}
{\beta_k\choose \alpha_k}$. 
An indexing set $\IND{n}$ is defined by
\begin{equation}\label{IdDef}
    \IND{n} = \left\{ \b{\alpha} \in \ZZ_+^{d+1}: \ |\b{\alpha}|=n \right\}.
\end{equation}
In this paper, 
we are mostly concerned with the three dimensional case $d=3$.
For $d=3$, let $\b{e}_\ell\in\INThree{1}$ ($\ell=1,2,3, 4$) denote the multi-index whose $\ell$-th entry is
unity and whose remaining entries vanish.

Let $T$ be a non-degenerate tetrahedron in $\RR^3$ given by 
$\mathrm{conv}(\b{x}_1,\b{x}_2,\b{x}_3,\b{x}_{4})$.
The set
$\c{D}_n(T)=\{\b{x}_\b{\alpha}: \b{\alpha}\in\INThree{n}\}$ consists of the
\emph{domain points} of the tetrahedron $T$ defined by
\begin{equation}
\label{DPTet}
    \b{x}_\b{\alpha} = \frac{1}{n}\sum_{k=1}^{4} \alpha_k\b{x}_k.
\end{equation}
We denote the 
edge-based, face-based, and cell-based bubble index subsets of $\INThree{n}$ as follows:
for an edge $E=(\b x_{e_1},\b x_{e_2})\subset T$,
\begin{subequations}
 \label{indices}
 \begin{align}
\mathring{  \c{I}}_3^n(E) = \{\b\alpha\in \c{I}_3^n:\quad 
  \alpha_{i} >0 \text{ if } i\in \{e_1,e_2\},  
  \;\; \alpha_{i} =0 \text{ otherwise}
  \}  
 \end{align}
for a face $F=(\b x_{f_1},\b x_{f_2},\b x_{f_3})\subset T$,
 \begin{align}
\mathring{  \c{I}}_3^n(F)&\; = \{\b\alpha\in \c{I}_3^n:\quad 
  \alpha_{i} >0 \text{ if } i\in \{f_1,f_2,f_3\}.
  \;\; \alpha_{i} =0 \text{ otherwise}
  \}
 \end{align}
 and for the element $T$,
 \begin{align}
\mathring{  \c{I}}_3^n&\; = \{\b\alpha\in \c{I}_3^n:\quad 
  \alpha_{i} >0 \text{ for all } i\},
 \end{align}
 The collection of these bubble index sets is denoted as 
 \begin{align}
  \check{\c I}_3^{n} :=&\; \INThree{n}-\{n\b e_{\ell}:\;\; 1\le \ell\le 4\}\\
=&\; \bigoplus_{E\in \mathcal{E}(T)}\mathring{  \c{I}}_3^n(E)
\bigoplus_{F\in \mathcal{F}(T)}\mathring{  \c{I}}_3^n(F)
\bigoplus\mathring{  \c{I}}_3^n.
\nonumber
 \end{align}
where $\mathcal{E}(T)$ and $\mathcal{F}(T)$ 
are the collection of edges and faces, respectively, of the tetrahedron $T$.

We also denote the set of indices whose corresponding domain points lie on 
the face $F = (\b x_{f_1},\b x_{f_2},\b x_{f_3})$ as
\begin{align}
\label{id-face-bubble}
 \tilde{\c{I}}^{n}_{3}(F):=\{\b \alpha\in \INThree{n}:\;
\alpha_{i}=0,\;\;\; i \not\in \{f_1,f_2,f_3\}\}.
\end{align}
Moreover, the index subsets  of co-dimension $1$ for
$\tilde{\c{I}}^{n}_{3}(F)$ and $\INThree{n}$ will also be used:
\begin{align}
\label{id-face-sigle}
 \tilde{\c{I}}^{n}_{3}(F)' := & \tilde{\c{I}}^{n}_{3}(F) - \{\text{a single, arbitrary index in } \tilde{\c{I}}^{n}_{3}(F)\},\\
 \label{id-cell-sigle}
 {{\c{I}}^{n}_{3}}' := & {\c{I}}^{n}_{3} - \{\text{a single, arbitrary index in } {\c{I}}^{n}_{3}\}.
\end{align}
\end{subequations}

\subsection{Barycentric coordinates and Bernstein polynomials}
The \emph{barycentric coordinates} of a point $\b{x}\in \RR^{3}$ with respect to
the tetrahedron $T$ are given by $\b{\lambda}=(\lambda_1,\lambda_2,\lambda_3,\lambda_{4})$
and satisfy
\begin{equation}\label{barycentricSum}
    \b{x}= \sum_{k=1}^{4} \lambda_k \b{x}_k;\quad 1 = \sum_{k=1}^{4} \lambda_k.
\end{equation}
The {\it Bernstein polynomials} of degree $n\in\ZZ_+$ associated with $T$ are defined
by
\begin{equation}\label{bernsteinPolyDef}
    B^{n}_\b{\alpha}(\b{x}) = {n\choose\b{\alpha}} \b{\lambda}^\b{\alpha},
    \quad\b{\alpha}\in\INThree{n}.
\end{equation}

\subsection{Geometric decomposition of the finite element spaces and their basis functions}
We follow the popular convention whereby basis functions are identified with dots placed at appropriate
domain points.
Our main result is collected in the following theorem, whose proof is postponed to 
Section \ref{sec:curl} and Section \ref{sec:div} where we study details of the 
$H(\mathrm{curl})$-basis and $H(\mathrm{div})$-basis, respectively.

\begin{theorem}
\label{thm:main}
Let $\mathcal{E}(T)$ be the collection of edges of $T$, and $E\in \mathcal{E}(T)$ be any edge.
Let $\mathcal{F}(T)$ be the collection of faces of $T$, and $F\in \mathcal{E}(T)$ be any face.
The sets of functionsdefined in Table \ref{table-span}
form a basis for the corresponding space for arbitrary polynomial order $n$. 

Moreover, 
the following geometric decomposition holds
for the finite element spaces in
the two sequences in
Table \ref{table-ex}:
\begin{align*}
  \pol_n &\;= \SSS_{low}\oplus_{E\in \mathcal{E}(T)}\SSS_n^E
 \oplus_{F\in \mathcal{F}(T)}\SSS_n^F\oplus \SSS_n^T\\
\pol_{n}^3
\oplus\b x
\times{\widetilde\pol}_{n}^3
 &\;= \EEE_{low}\oplus_{E\in \mathcal{E}(T)}\Grad\SSS_{n+1}^E
 \oplus_{F\in \mathcal{F}(T)}
\left( \Grad\SSS_{n+1}^F\oplus
\EEE_n^F
\right)
 \oplus \left(\Grad\SSS_{n+1}^T
 \oplus
\EEE_{n+1}^T\right)\\
\pol_{n}^3
 &\;= \EEE_{low}\oplus_{E\in \mathcal{E}(T)}\Grad\SSS_{n+1}^E
 \oplus_{F\in \mathcal{F}(T)}
\left( \Grad\SSS_{n+1}^F\oplus
\EEE_{n-1}^F
\right)
 \oplus \left(\Grad\SSS_{n+1}^T
 \oplus
\EEE_{n}^T\right)\\
\pol_{n}^3
\oplus \b x\,\widetilde\pol_n
 &\;= \VVV_{low} \oplus_{F\in \mathcal{F}(T)}
\Curl\EEE_n^F
 \oplus \left(\Curl
\EEE_{n+1}^T
\oplus
\VVV_n^T
\right)\\
\pol_{n}^3
 &\;= \VVV_{low} \oplus_{F\in \mathcal{F}(T)}
\Curl\EEE_n^F
 \oplus \left(\Curl
\EEE_{n+1}^T
\oplus
\VVV_{n-1}^T
\right)\\
 \pol_{n} &\;= \WWW_{low} 
 \oplus\Div\VVV_n^T
\end{align*}
\end{theorem}

\begin{table}[!ht]
\caption{Local spaces and their basis.
 Here $\b\omega_{ij}$ is the lowest-order edge element given in  \eqref{ND3D-lowestorder} in Section \ref{sec:curl},
 $\b\chi_\ell$ is the lowest-order face element given in \eqref{DIV3D-lowestorder}
 in Section \ref{sec:div}, $\Xa{FT,n}{\b\alpha}$ 
  is   the $\b H(\mathrm{curl})$-face bubble given in \eqref{BB-facebubble} and 
$\Xb{T,n}{\ell,\b\alpha}$ is the $\b H(\mathrm{curl})$-cell bubble given in \eqref{BB-cellbubble}
  in Section \ref{sec:curl}, and $\Xc{n}{\b\alpha}$ 
  is 
  the $\b H(\mathrm{div})$-bubble function given in \eqref{BB-divbubble} in Section \ref{sec:div}. 
}
\begin{tabular}{r  l | r l}
\hline
\noalign{\smallskip}
\multicolumn{4}{c}{Lowest-order elements}\\
\hline
$ {\SSS_{{low}}}:$ &$\mathrm{span}\{\lambda_i:\; 1\le i\le 4\}$&
$ {\EEE_{{low}}} :$&$\mathrm{span}\{\b\omega_{ij}:\;1\le i<j\le 4\}$\vspace{.1cm}\\
$ {\VVV_{{low}}} :$&$\mathrm{span}\{\b\chi_\ell:\; 1\le \ell\le 4\}$&
$ {\WWW_{{low}}} :$&$\mathrm{span}\{1\}$\\
\hline
\noalign{\smallskip}
\multicolumn{4}{c}{$H^1$ edge, face, cell bubbles and their gradients}\\
\hline
 ${\SSS^E_{{n}}} :$ &$\;\mathrm{span}\{  \BB{n}{\b{\alpha}}:\; \b\alpha\in 
 \mathring{\c I}_3^n(E) \}$&
$ \Grad{\SSS^E_{{n}}} :$&$\mathrm{span}\{  \Grad\BB{n}{\b{\alpha}}:\; \b\alpha\in 
 \mathring{\c I}_3^n(E) \}$\vspace{.1cm}\\
$ {\SSS^F_{{n}}} :$&$\;\mathrm{span}\{  \BB{n}{\b{\alpha}}:\;\b\alpha\in 
 \mathring{\c I}_3^n(F) \}$&
$ \Grad{\SSS^F_{{n}}} :$&$\;\mathrm{span}\{  \Grad\BB{n}{\b{\alpha}}:\; \b\alpha\in 
 \mathring{\c I}_3^n(F) \}$\vspace{.1cm}\\
$ {\SSS^T_{{n}}} :$&$\;\mathrm{span}\{  \BB{n}{\b{\alpha}}:\;\b\alpha\in 
 \mathring{\c I}_3^n \}$&
 $\Grad{\SSS^T_{{n}}} :$&$\;\mathrm{span}\{  \Grad\BB{n}{\b{\alpha}}:\; \b\alpha\in 
 \mathring{\c I}_3^n \}$\\
\hline
\noalign{\smallskip}
\multicolumn{4}{c}{$\b H(\mathrm{curl})$ face, cell bubbles and their curls}\\
\hline
$ {\EEE^F_{{n}}} :$&$\;\mathrm{span}\{ \Xa{FT,n}{\b\alpha}:\;\b\alpha\in 
 \tilde{\c I}_3^n(F)' \}$&
$ \Curl{\EEE^F_{{n}}} :$&$\;\mathrm{span}\{  \Curl\Xa{FT,n}{\b\alpha}:\; \b\alpha\in 
 \tilde{\c I}_3^n(F)' \}$\\
\noalign{\smallskip}
$ {\EEE^T_{{n}}} :$&$\;\mathrm{span}\oplus_{\ell=1}^2\{
 \Xb{T,n}{\ell,\b\alpha}:\;\b\alpha\in 
 \mathring{\c I}_3^{n+1}\}$
 &$\Curl {\EEE^T_{{n}}} :$&$\;\mathrm{span}\oplus_{\ell=1}^2\{
 \Curl\Xb{T,n}{\ell,\b\alpha}:\;\b\alpha\in 
 \mathring{\c I}_3^{n+1}\}$
\\
&$ \quad\oplus\{
 \Xb{T,n}{3,\b\alpha}:\;\b\alpha\in 
 \mathring{\c I}_3^{n+1},\alpha_3 = 1\}$&
&$\quad\oplus\{
\Curl \Xb{T,n}{3,\b\alpha}:\;\b\alpha\in 
 \mathring{\c I}_3^{n+1},\alpha_3 = 1\}$\\
 \hline
 \noalign{\smallskip}
\multicolumn{4}{c}{$\b H(\mathrm{div})$ cell bubbles and their divergences}\\
\hline
$ {\VVV^T_{{n}}} :$&$\;\mathrm{span}\{  
 \Xc{n}{\b\alpha}:\;\b\alpha\in 
 {\c I}_3^{n'} \}$&
$ \Div{\VVV^T_{{n}}} :$&$\;\mathrm{span}\{  
\Div \Xc{n}{\b\alpha}:\;\b\alpha\in 
 {\c I}_3^{n'} \} $\\
 \hline
\end{tabular}
\label{table-span}
\end{table}

We illustrate
the result in Theorem \ref{thm:main}
for each sequence in Table \ref{table-geo} and indicate the dimension of the spaces for each 
case.

We also collect different types of basis functions and the corresponding dimension counts for the 
$\b H(\mathrm{curl})$-space
$
\NDOne{n}
$ 
and the $\b H(\mathrm{div})$-space
$\RT{n}$, along with the domain points (for   $n=3$) in Table \ref{table-dp-curl} and 
Table \ref{table-dp-div}, respectively. 
The index corresponding to the top white circle (for Type 3 $\b H(\mathrm{curl})$-bases, and 
Type 2 and Type 4 $\b H(\mathrm{div})$-bases) in the domain points plots  indicates 
arbitrarily chosen index used in the definition \eqref{id-face-sigle}--\eqref{id-cell-sigle}.
Here for the Type 3 $\b H(\mathrm{curl})$-face bubbles, we only plot the domain points for 
{\it one} face to make the plot more readable. 
Also, a domain point for 
Type 4 $\b H(\mathrm{curl})$-cell bubbles corresponds to $2$ basis functions ($\Xb{n+1}{1,\b\alpha}, \Xb{n+1}{2,\b\alpha}$) if $\alpha_3 >1$, and $3$ basis functions 
($\Xb{n+1}{1,\b\alpha}, \Xb{n+1}{2,\b\alpha}, \Xb{n+1}{3,\b\alpha}$) if 
$\alpha_3 = 1$. 
The same information is presented for
$\b H(\mathrm{div})$-bases.

In order to visualize the nature of the bases, we present contour plots
of some representative basis functions in Figure \ref{fig:curl} and Figure \ref{fig:div}.
  
\begin{table}[!ht]
\caption{Bernstein-B\'ezier basis for the two sequences in 
Table \ref{table-ex}.
Here $\delta = 0$ for the first sequence and 
$\delta = 1$ for the second sequence.
}
\begin{tabular}{l l l l l l l l l}
\vspace{0.04cm}
\sf{low}&
$\SSS_{low}$&
$\!\!\!\!\overset{\Grad}{\longrightarrow}$ 
&
$\EEE_{low}$&
$\!\!\!\!\overset{\Curl}{\longrightarrow}$ 
&
$\VVV_{low}$&
$\!\!\!\!\overset{\Div}{\longrightarrow}$ 
&
$\WWW_{low}$\\
\noalign{\smallskip}
Ndofs& 
$4$&&
$6$&& 
$4$&&
$1$
\\
\hline
\noalign{\smallskip}
\sf{edge}&
$\SSS_{n+1+\delta}^E$&
$\!\!\!\!\overset{\Grad}{\longrightarrow}$ 
&
$\Grad\SSS_{n+1+\delta}^E$&
$\!\!\!\!\overset{\Curl}{\longrightarrow}$ 
&$\{0\}$\\
\noalign{\smallskip}
Ndofs& 
$n+\delta$&&
$n+\delta$&& 
&&
\\
\hline
\noalign{\smallskip}
\sf{face}&
$\SSS_{n+1+\delta}^F$&
$\!\!\!\!\overset{\Grad}{\longrightarrow}$ 
&
$\Grad\SSS_{n+1+\delta}^F\oplus 
\EEE_{n}^F
$&
$\!\!\!\!\overset{\Curl}{\longrightarrow}$ 
&
$\Curl\EEE_{n}^F$&
$\!\!\!\!\overset{\Div}{\longrightarrow}$ 
&
$\{0\}$
\\
\noalign{\smallskip}
Ndofs& 
${n+\delta\choose 2}$
&&
${n+\delta\choose 2}+
{n+2\choose 2}-1
$&& 
${n+2\choose 2}-1$
\\
\hline
\noalign{\smallskip}
\sf{cell}&
$\SSS_{n+1+\delta}^T$&
$\!\!\!\!\overset{\Grad}{\longrightarrow}$ 
&
$\!\!\!\!\Grad\SSS_{n+1+\delta}^T\oplus 
\EEE_{n+1}^T
$&
$\!\!\!\!\overset{\Curl}{\longrightarrow}$ 
&
$\!\!\!\!\Curl\EEE_{n+1}^T\oplus 
\VVV_{n-\delta}^T
$&
$\!\!\!\!\overset{\Div}{\longrightarrow}$ 
&
$\!\!\!\!\Div\VVV_{n-\delta}^T$\\
\noalign{\smallskip}
Ndofs& 
${n+\delta\choose 3}$
&&
$\hspace{-0.6cm}{n+\delta\choose 3}+
2{n+1\choose 3}\!+\!
{n\choose 2}
$& 
&$
\hspace{-0.6cm}
2{n+1\choose 3}\!+\!
{n\choose 2}
\!+\! {n+3-\delta\choose 3}$&$\hspace{-0.4cm}-1$
&
${n+3-\delta\choose 3}-1$
\end{tabular}
\label{table-geo}
\end{table}

\begin{table}[!ht]
\caption{Bases and domain points for the  $\b H(\mathrm{curl})$-bases for the space 
$\NDOne{n}$
}
\begin{tabular}{l c c c c}
& Type 1 & Type 2 & Type 3 & Type 4\\
&low & gradients &
non-gradient & non-gradient\\
&& &face bubbles & cell bubbles\\
\hline
\noalign{\smallskip}
basis&  $\b \chi_i$ & $\Grad \BB{n}{\b\alpha}$&
$ \Xa{FT,n}{\b\alpha}$ 
&
$ \Xb{n+1}{\ell,\b\alpha}$\\
\noalign{\smallskip}
\hline
\noalign{\smallskip}
Ndofs& $6$ &$ {n+4\choose 3}-4$ & 
$4{n+2\choose 2}-4$ &
$2{n+1\choose 3}+{n\choose 2}$
\\
\noalign{\smallskip}
\hline
\noalign{\smallskip}
$\overset{\text{domain points}}{n = 3}$& 
\scalebox{0.12}{\includegraphics{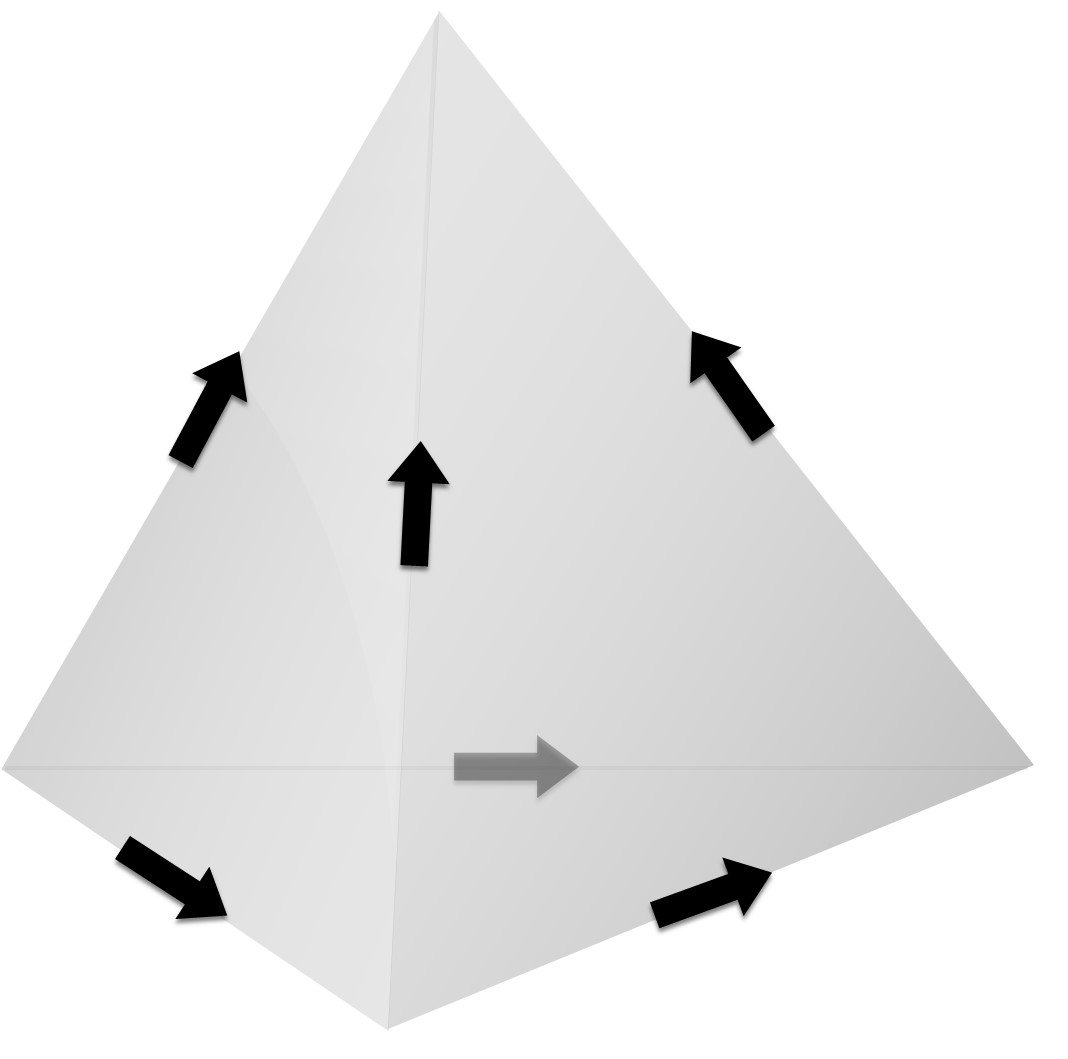}}&
\scalebox{0.12}{\includegraphics{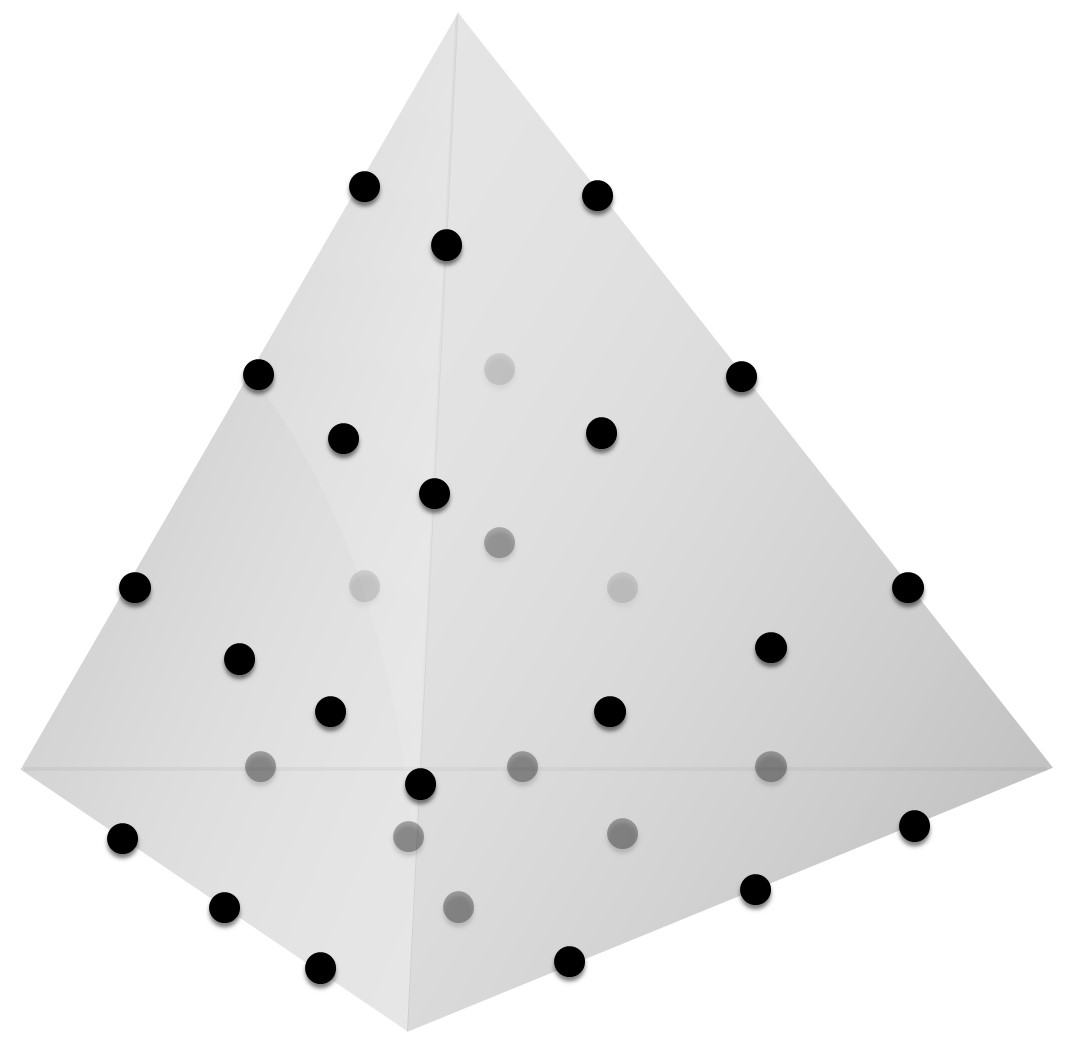}}&
\scalebox{0.12}{\includegraphics{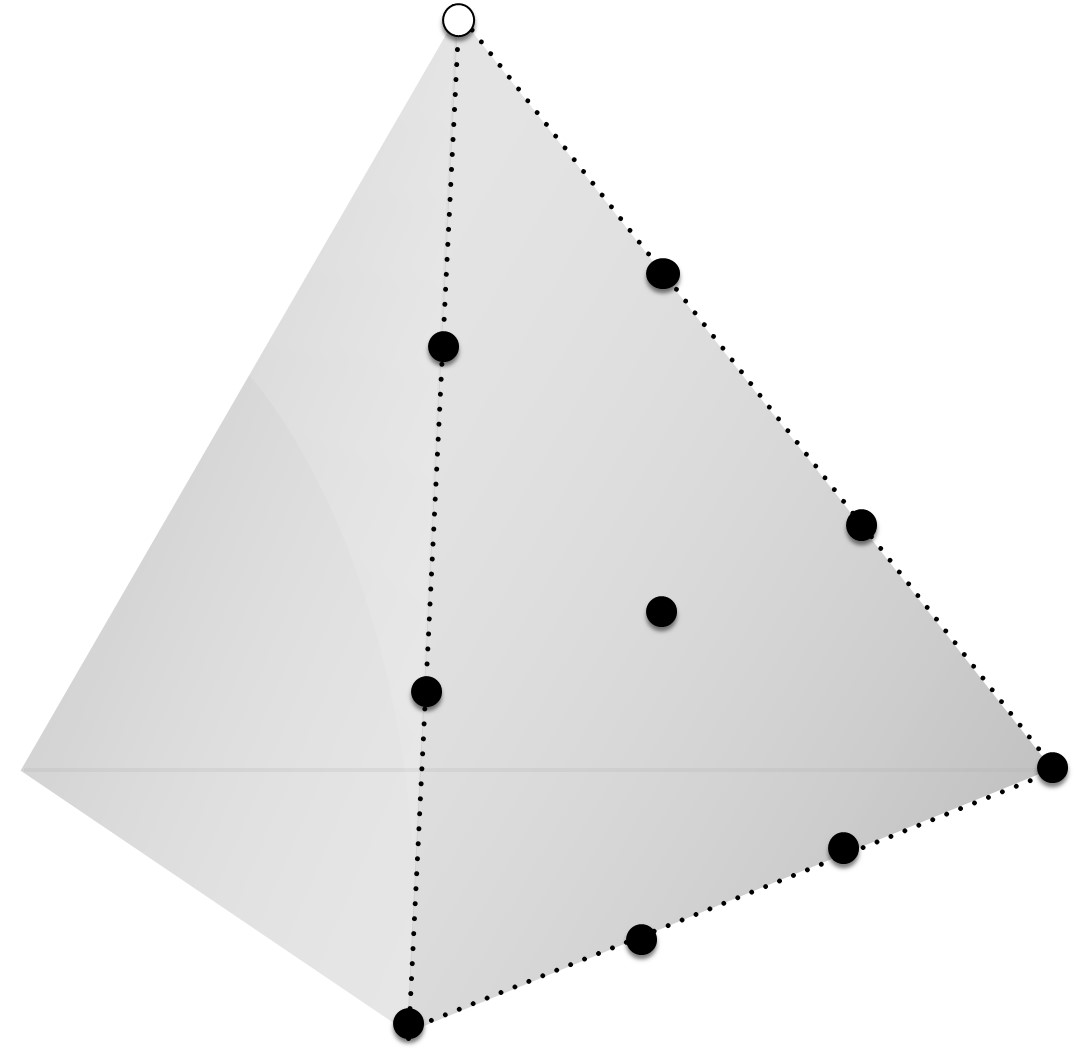}} 
&
\scalebox{0.12}{\includegraphics{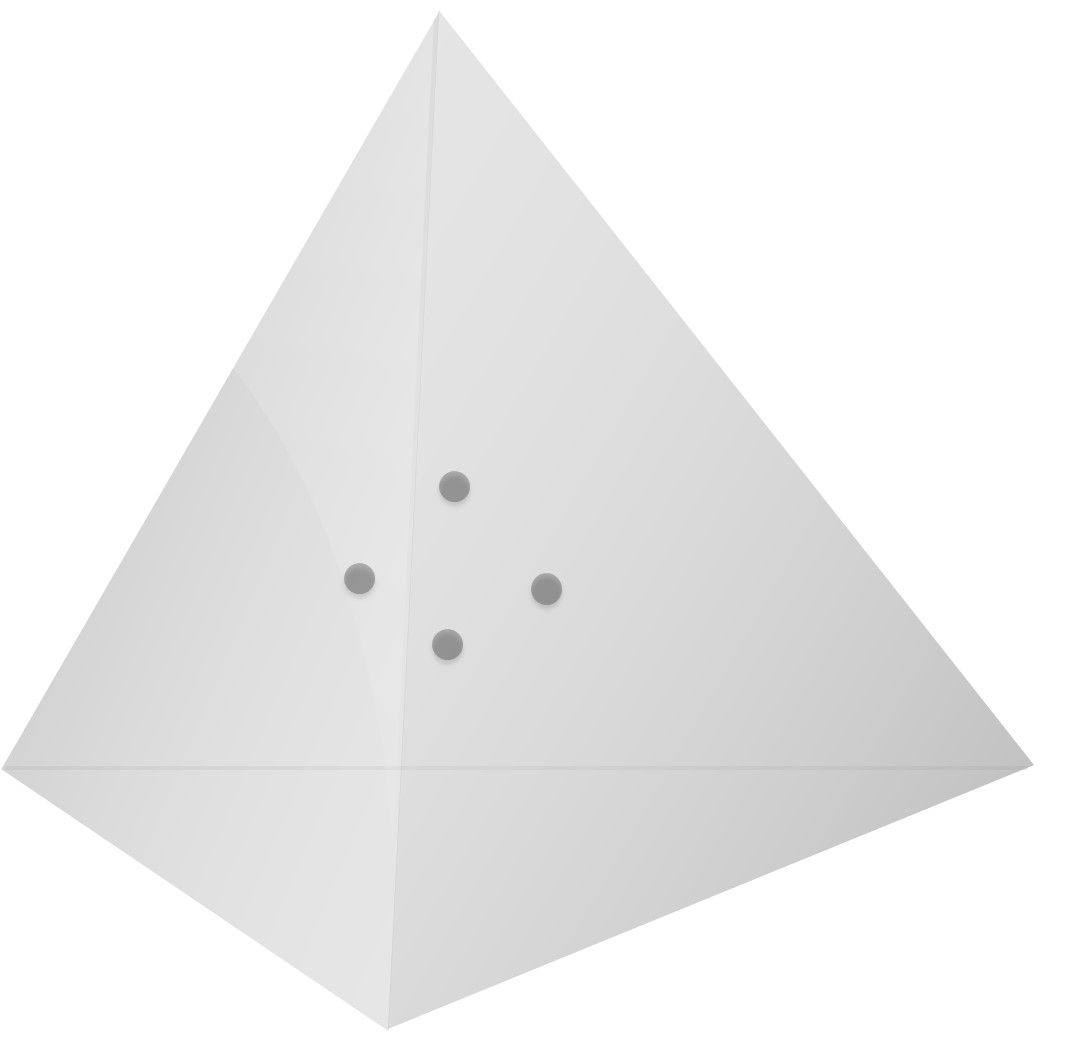}}
\\
\hline
 \end{tabular}
 \label{table-dp-curl}
\end{table}

\begin{table}[!ht]
\caption{Bases and domain points for the  $\b H(\mathrm{div})$-bases for the space 
$\RT{n}$
}
\begin{tabular}{l c c c c}
& Type 1 & Type 2 & Type 3 & Type 4\\
&low & div-free &
div-free & non-curl\\
& & face bubbles &
cell bubbles & cell bubbles\\
\hline
\noalign{\smallskip}
basis&  $\b \chi_i$ &
$\Curl \Xa{FT,n}{\b\alpha}$ 
&
$\Curl \Xb{n+1}{\ell,\b\alpha}$
& 
$\Xc{n}{\b\alpha}$\\
\noalign{\smallskip}
\hline
\noalign{\smallskip}
Ndofs& $4$ & 
$4{n+2\choose 2}-4$ &
$2{n+1\choose 3}+{n\choose 2}$
&
${n+3\choose 3}-1$
\\
\noalign{\smallskip}
\hline
\noalign{\smallskip}
$\overset{\text{domain points}}{n = 3}$& 
\scalebox{0.12}{\includegraphics{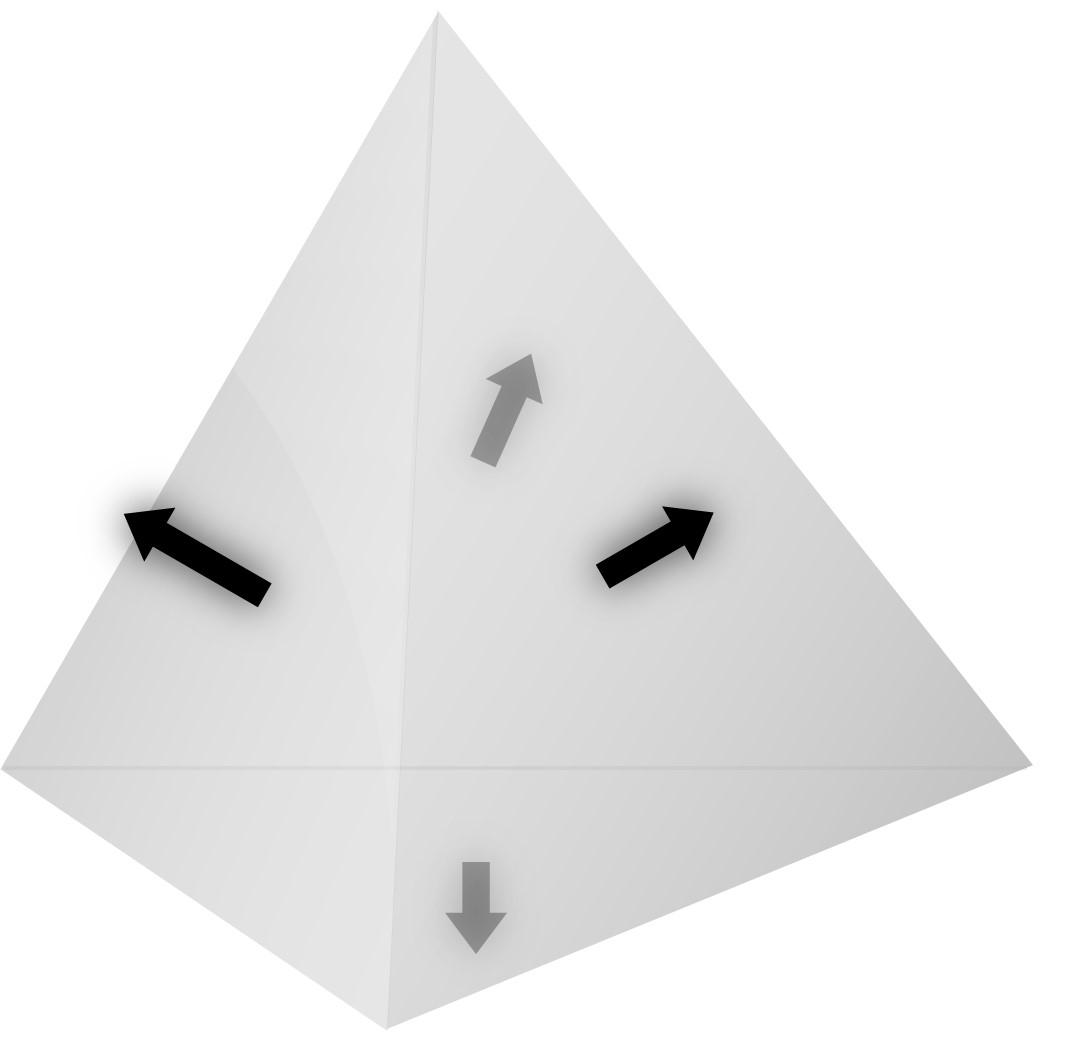}}&
\scalebox{0.12}{\includegraphics{hcurlf3.jpg}}&
\scalebox{0.12}{\includegraphics{hcurlc.jpg}}&
\scalebox{0.12}{\includegraphics{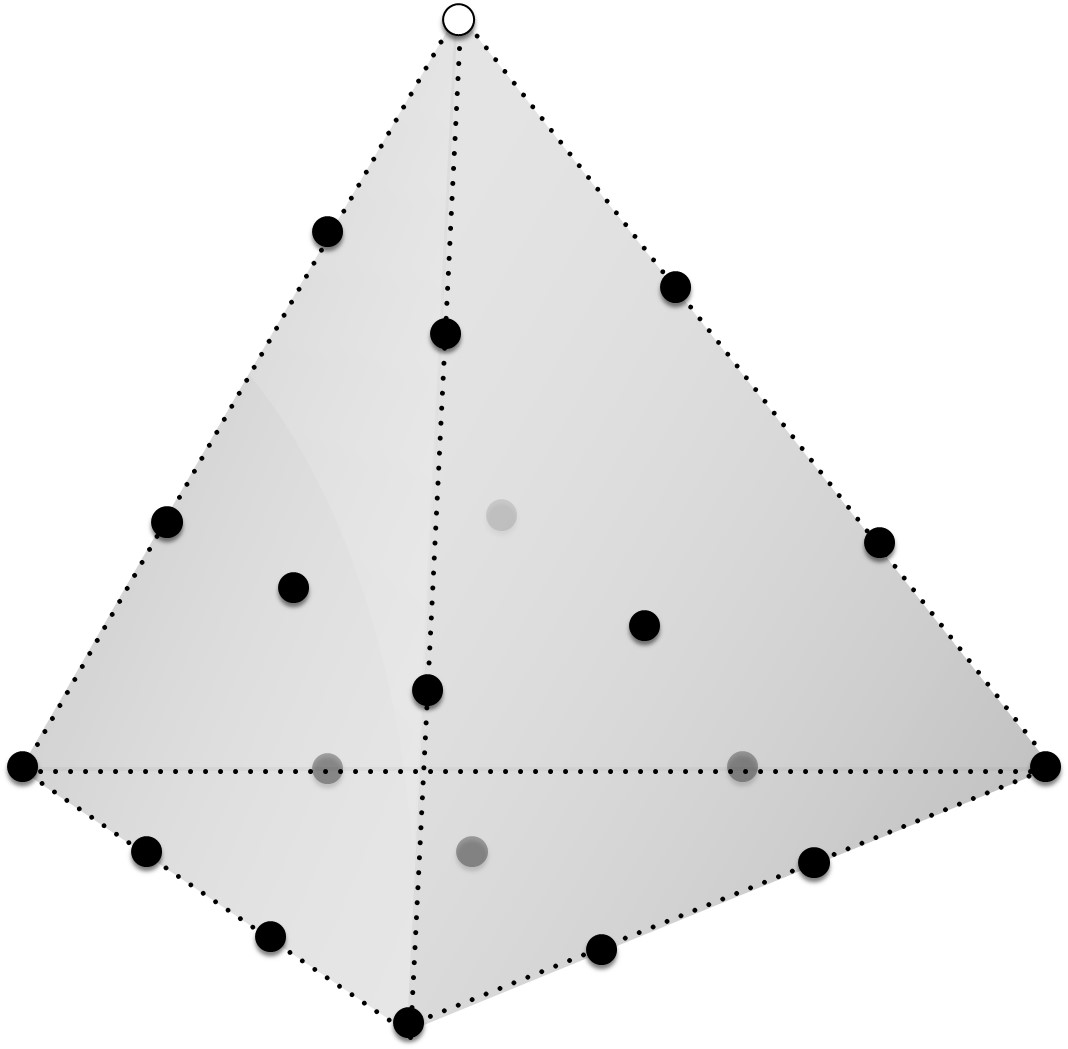}}
\\
\hline
 \end{tabular}
 \label{table-dp-div}
\end{table}

\begin{figure}
\includegraphics[width=.32\textwidth]{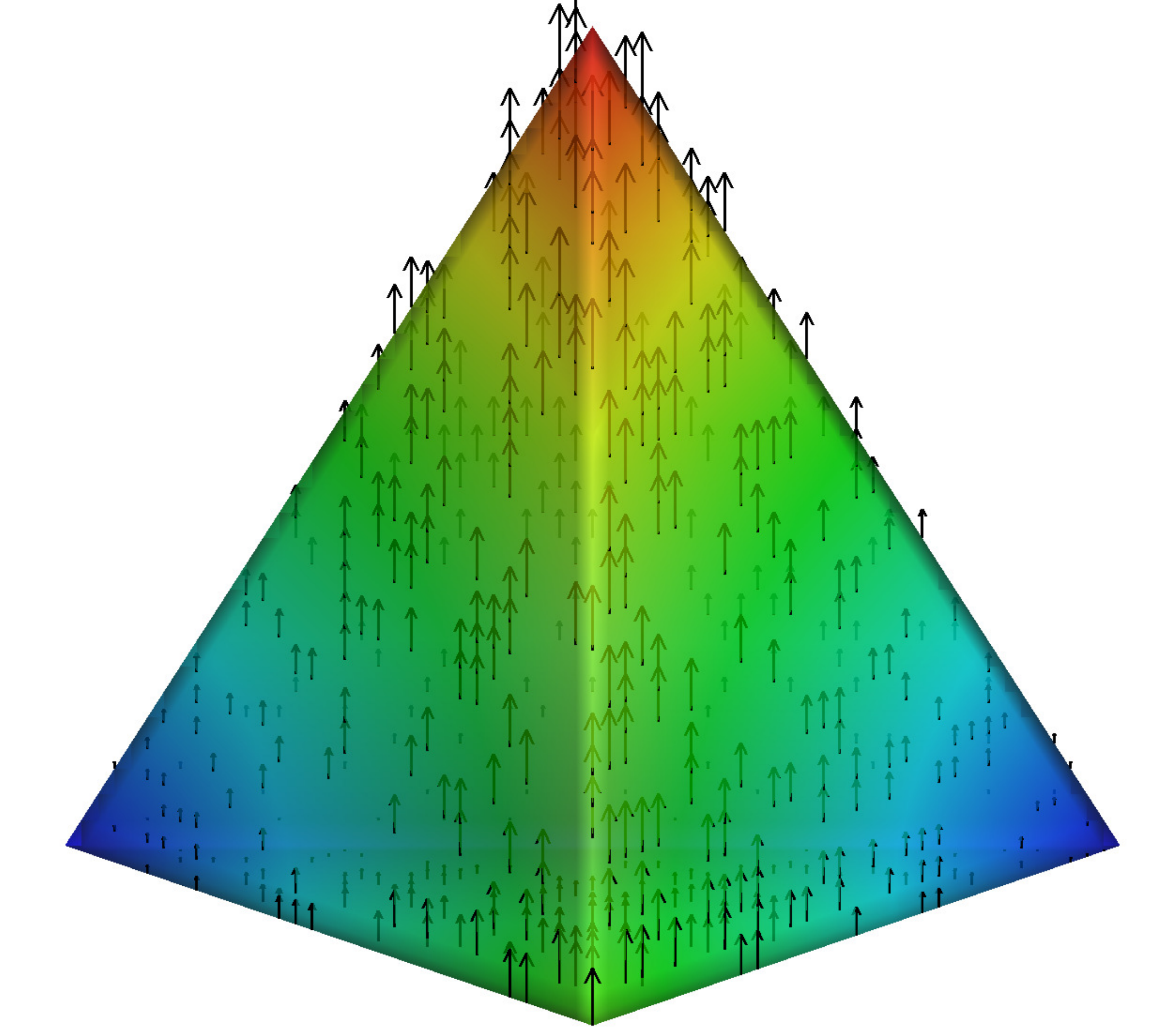}
\includegraphics[width=.32\textwidth]{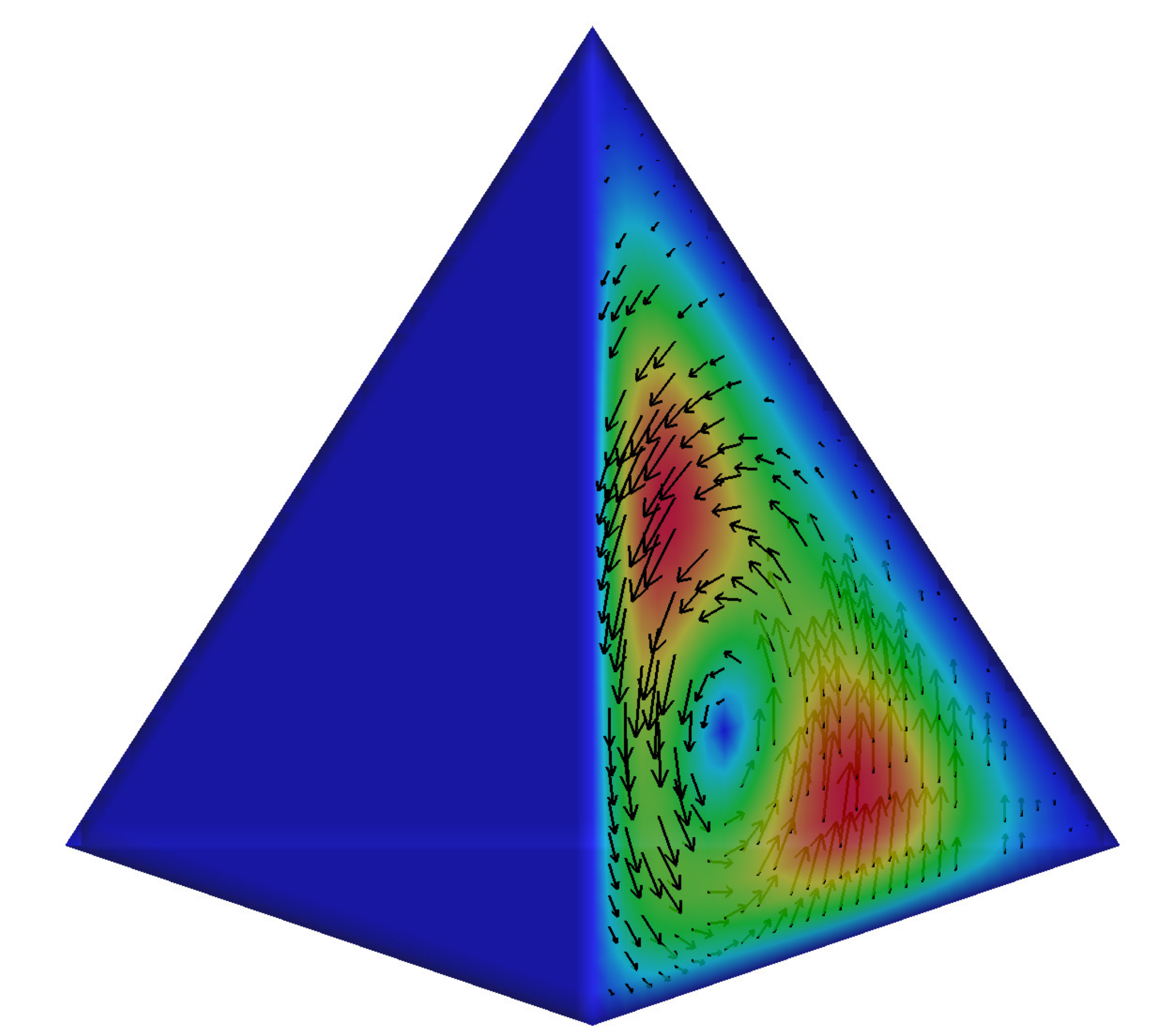}
\includegraphics[width=.32\textwidth]{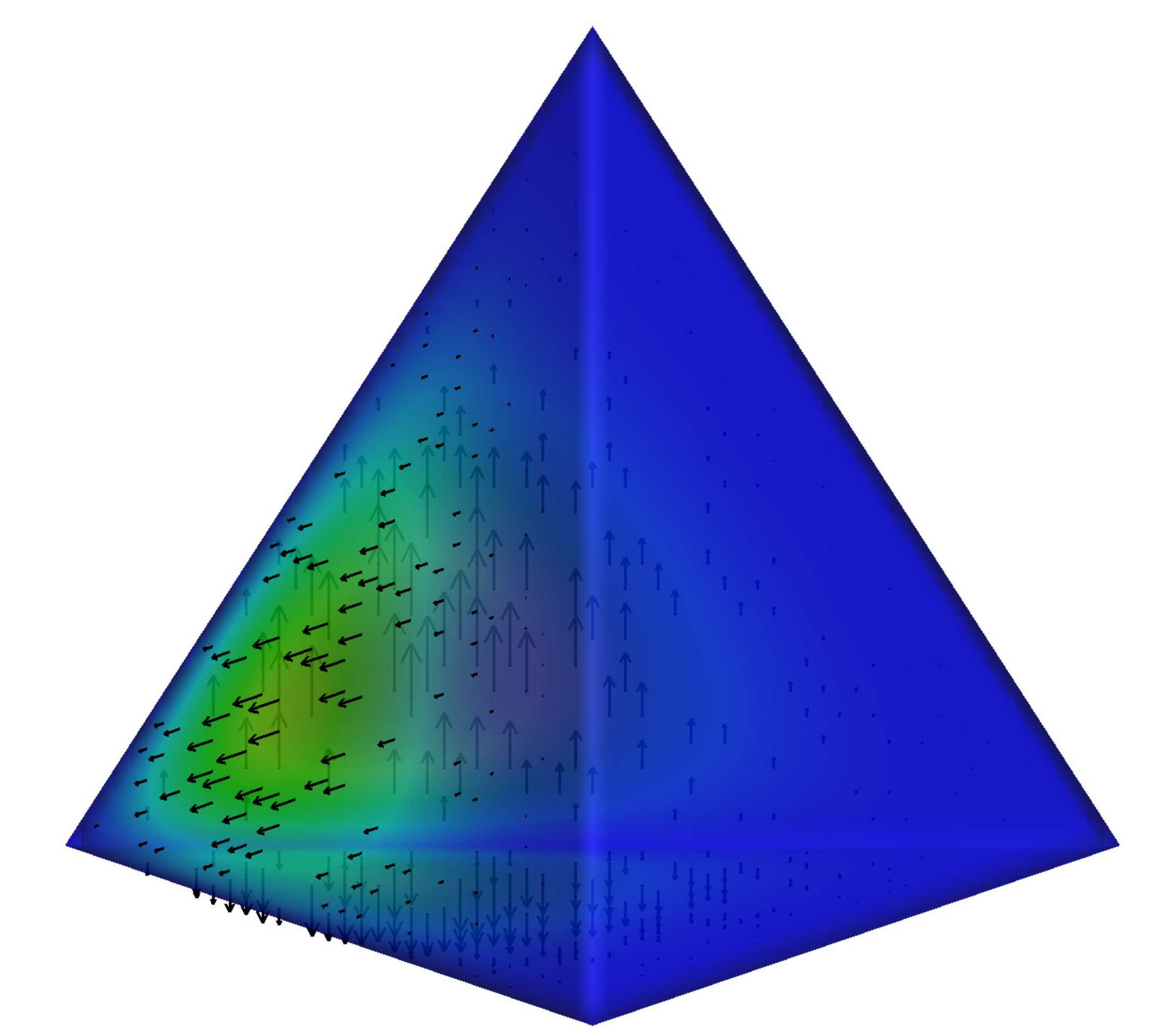}\\
\includegraphics[width=.32\textwidth]{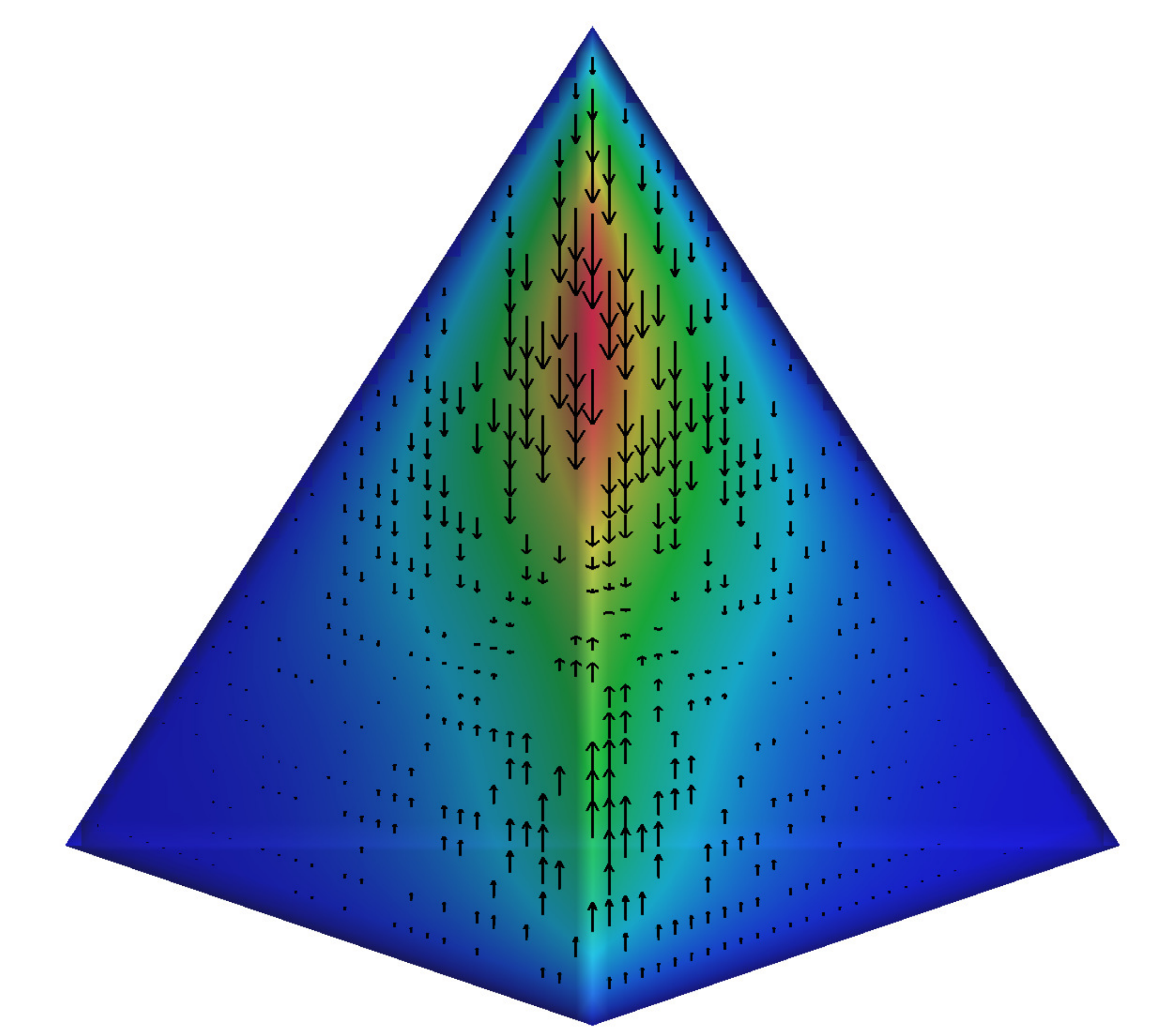}
\includegraphics[width=.32\textwidth]{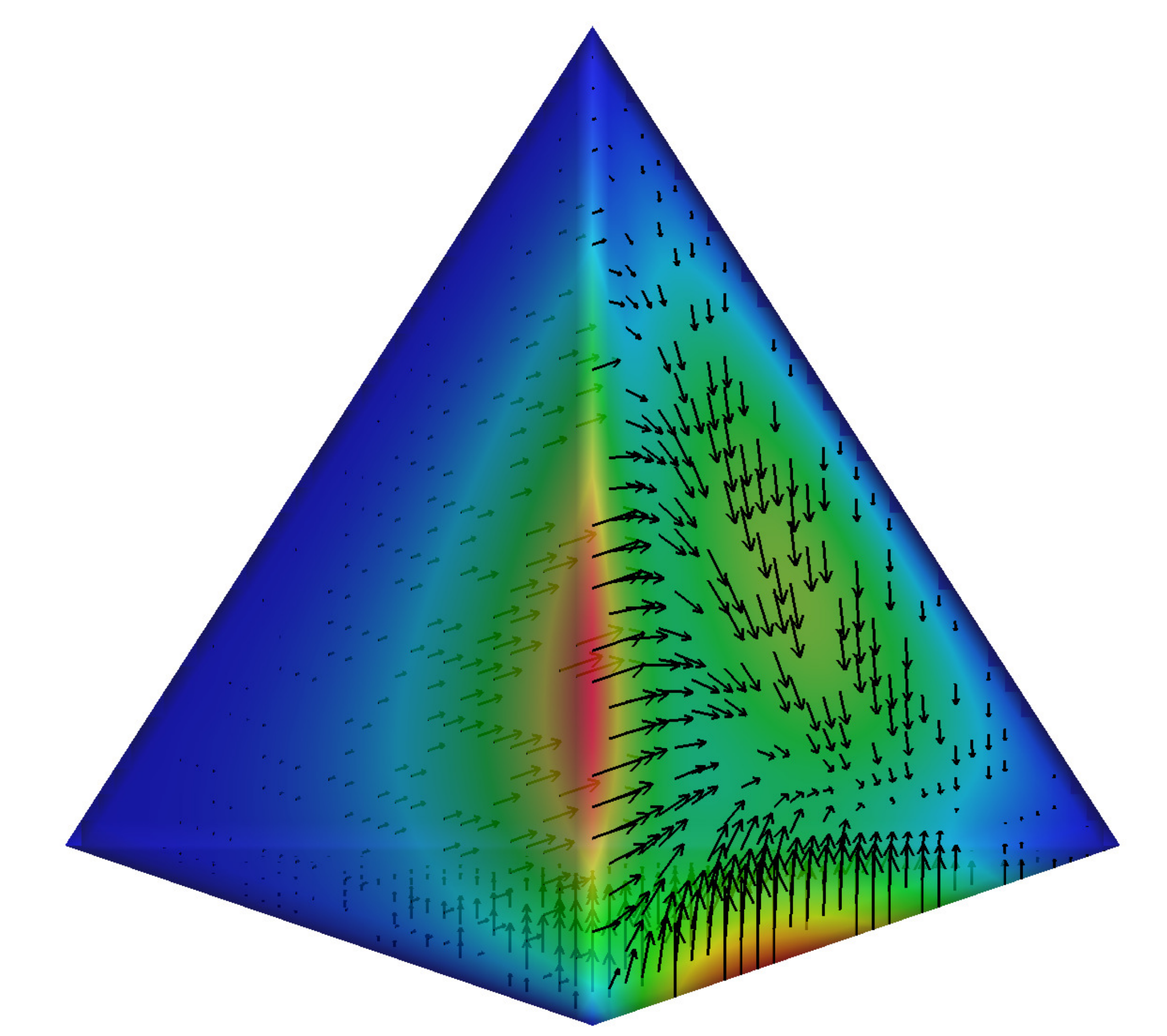}
\includegraphics[width=.32\textwidth]{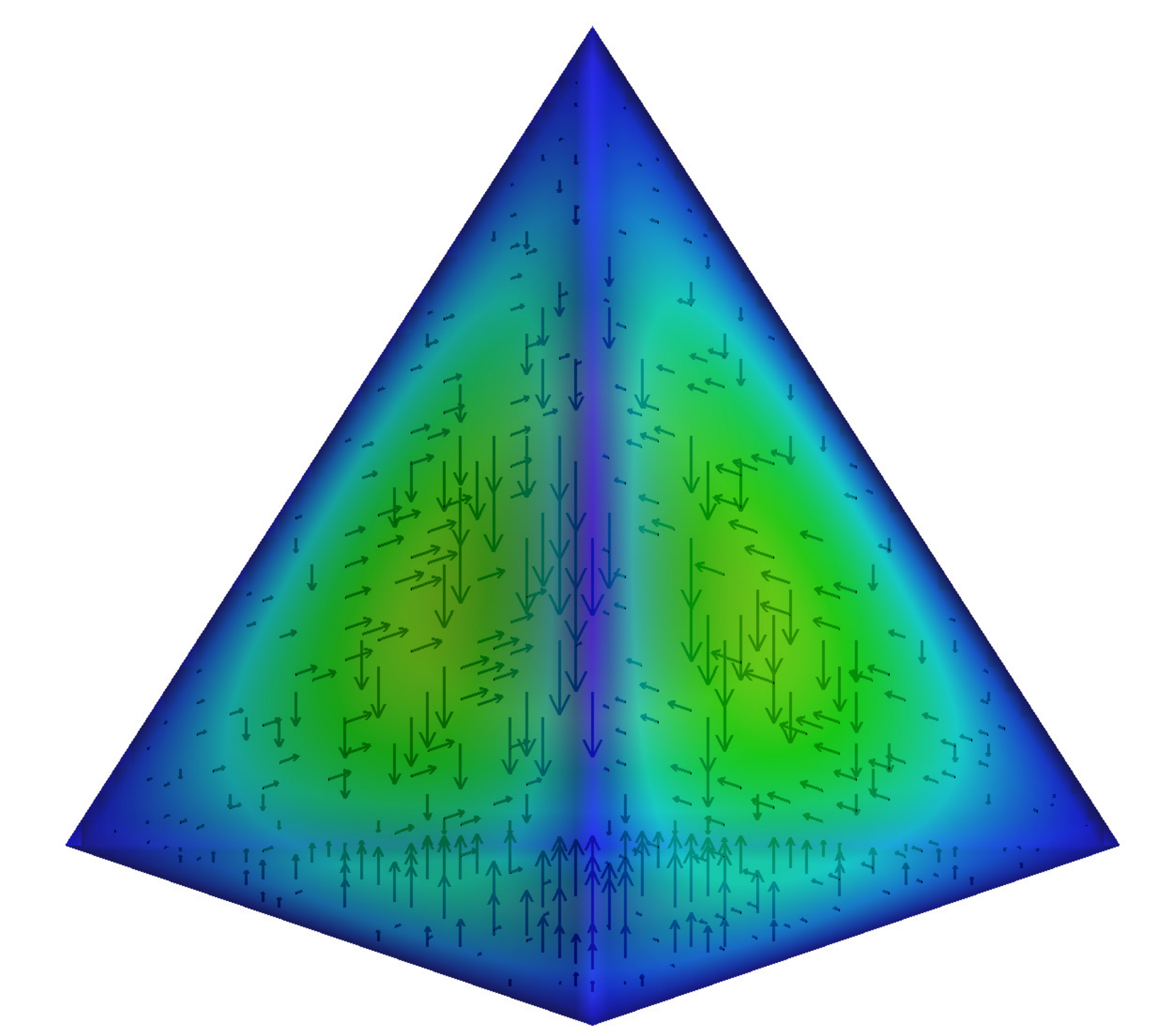}
\caption{Top:  left $\b\omega_{24}$, 
middle $\Xa{FT,4}{(0,2,1,1)}$, right: 
$\Xb{T,4}{1,(2,1,1,1)}$.
Bottom: left $\Grad \BB{4}{(0,2,0,2)}$, 
 middle $\Grad \BB{4}{(0,2,1,1)}$, 
 right $\Grad \BB{4}{(1,1,1,1)}$.
}
\label{fig:curl}
\end{figure}

\begin{figure}
\includegraphics[width=.32\textwidth]{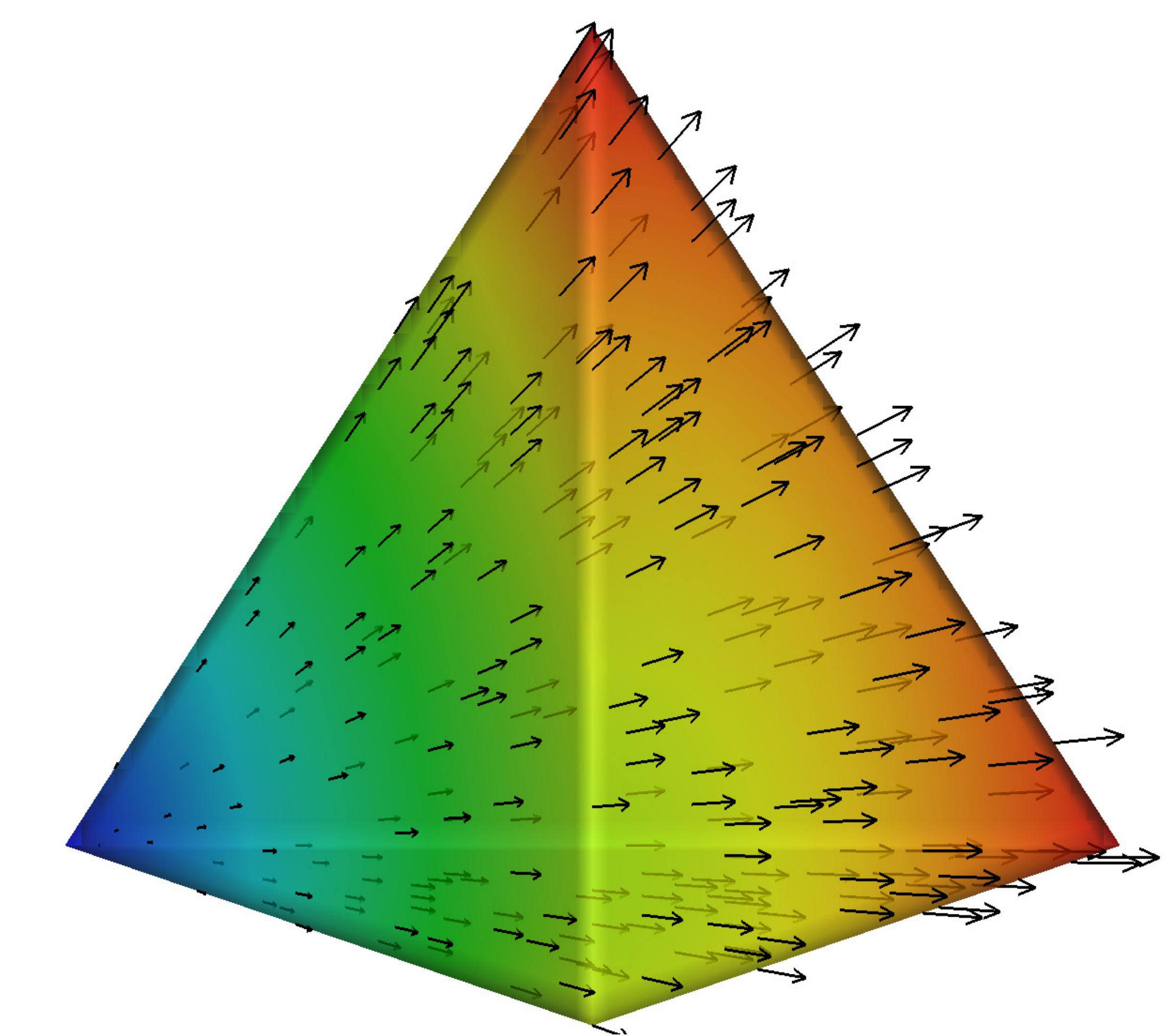}
\includegraphics[width=.32\textwidth]{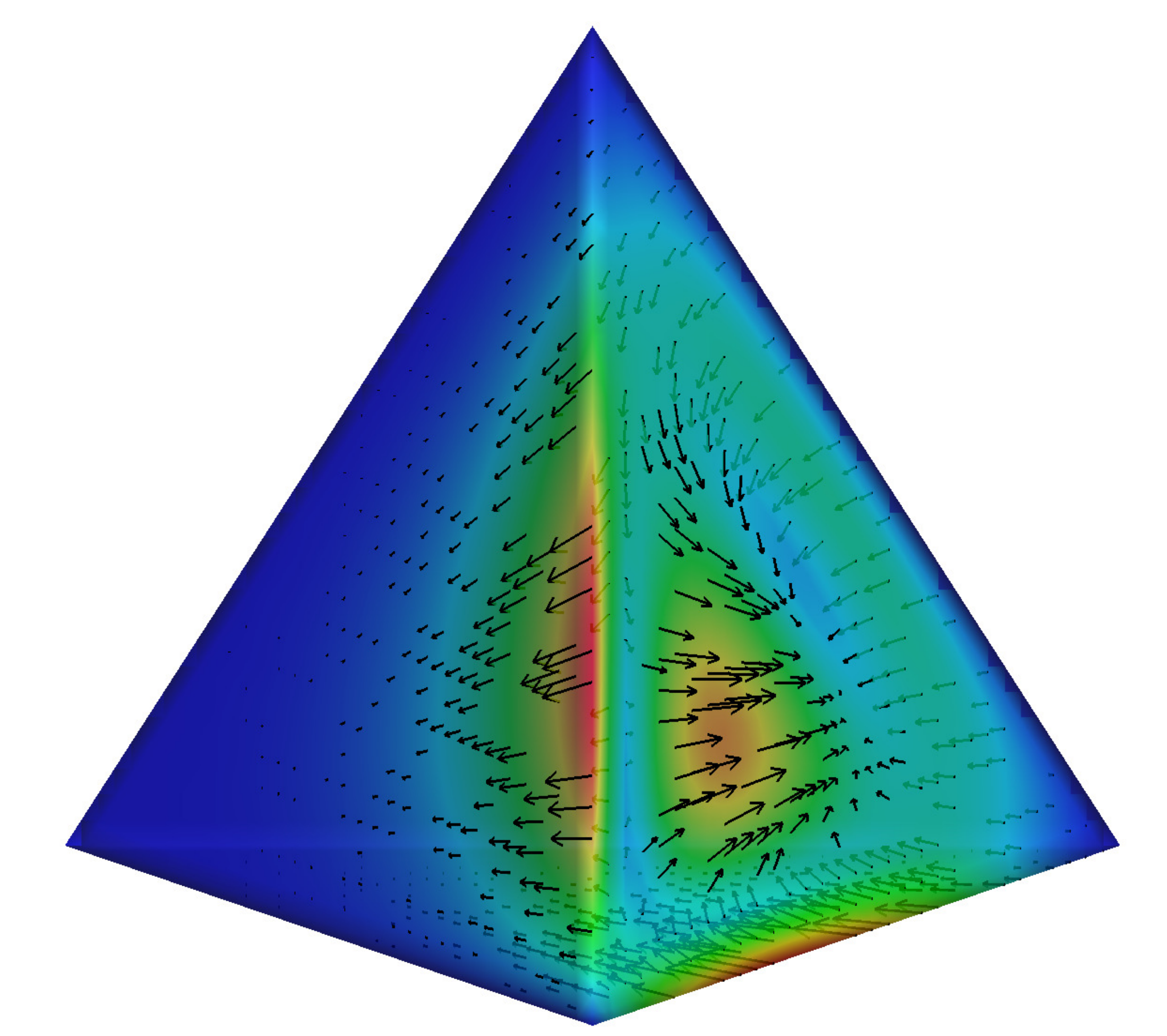}
\includegraphics[width=.32\textwidth]{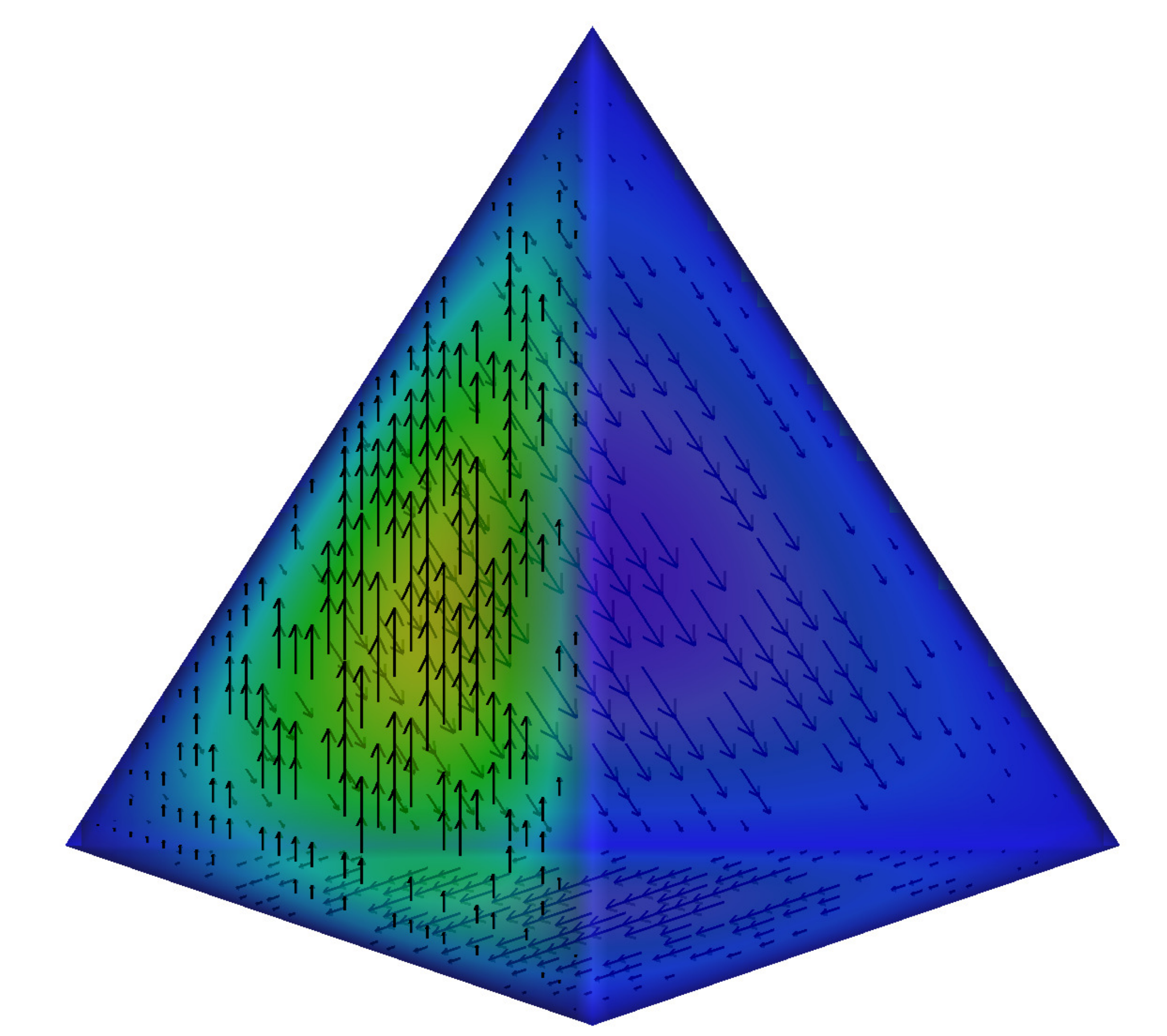}\\
\includegraphics[width=.32\textwidth]{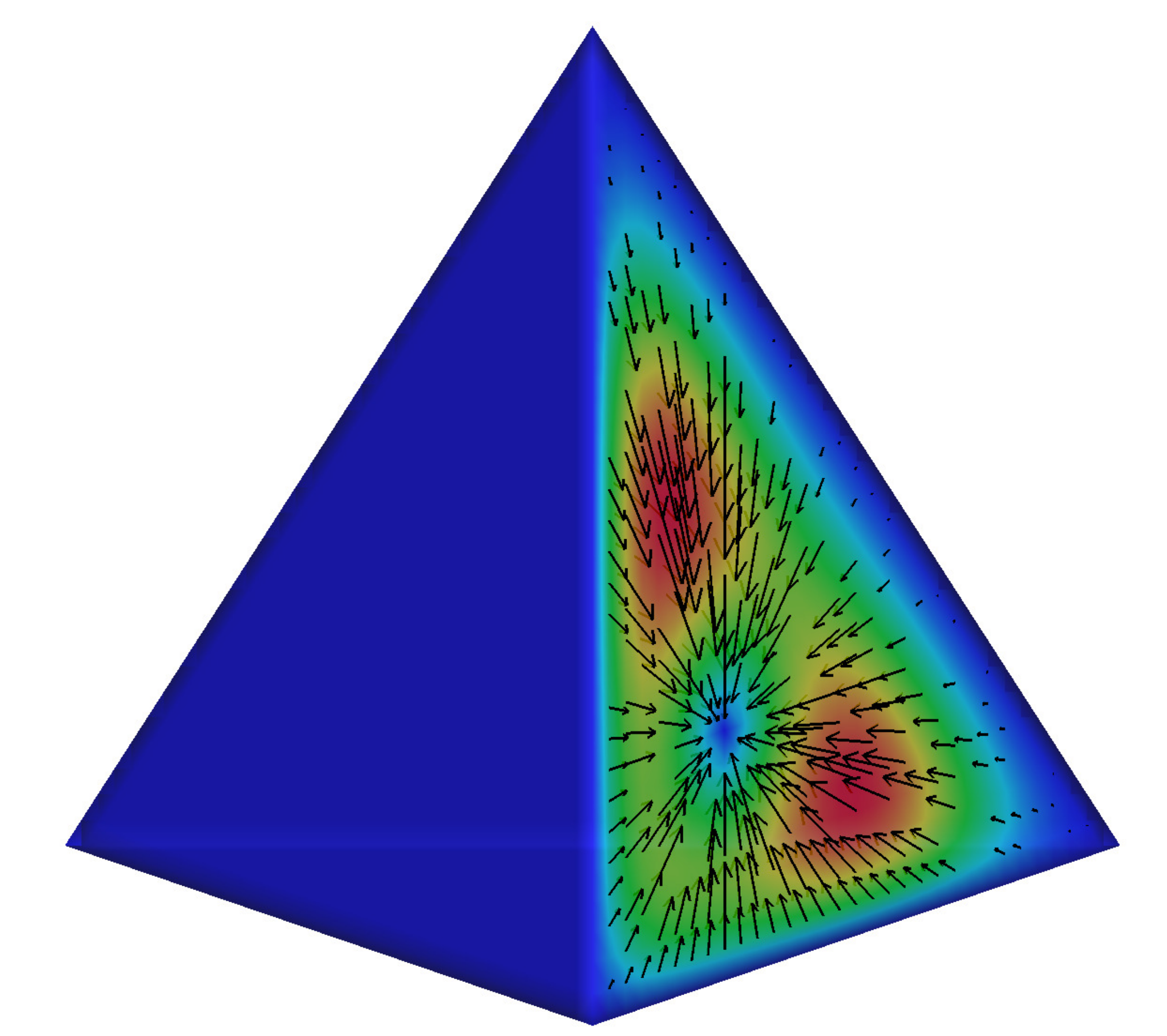}
\caption{Top:  left $\b\chi_{1}$, 
middle $\Curl\Xa{FT,4}{(0,2,1,1)}$, right: 
$\Curl\Xb{T,4}{1,(2,1,1,1)}$.
Bottom: $\Xc{4}{(0,2,1,1)}$.
}
\label{fig:div}
\end{figure}

%

Although the basis are presented for the case where the local order is uniform.
We may vary the polynomial degree associated to each edge, face and cell
without compromising the local exact sequence property.

\newcommand{\xdownarrow}[1]{%
  {\left\downarrow\vbox to #1{}\right.\kern-\nulldelimiterspace}
}

\newpage

\subsection{Embedded Sequences and Dimension Hierarchy}
Our basis also has the property that 
polynomial exact sequences 
on triangular elements which, in turn, contain exact sequences on edges, 
are 
``embedded'' in our tetrahedral basis.
These properties are illustrated in Figure \ref{fig:emb}, 
where the vertical arrows indicate the appropriate trace operator onto the lower dimensional entity.
\begin{figure}[ht!]
 \caption{Embedded Sequence}
\begin{tabular}{l c c c c c l l}
\noalign{\smallskip}
\multicolumn{8}{c}{First Sequence}\\
\hline
Tetrahedron: &$\pol_{n+1}
$&
$\overset{\Grad}{\longrightarrow}$ &
$
\pol_{n}^3
\oplus\b x
\times{\widetilde\pol}_{n}^3
$
&
$\overset{\Curl}{\longrightarrow}$ &
$
\pol_{n}^3
\oplus \b x\,\widetilde\pol_n
$&
$\overset{\Div}{\longrightarrow}$ &
$\pol_{n}
$\\
&$\xdownarrow{0.5cm}{\mathrm{tr}}$ & &
$\xdownarrow{0.5cm}{\mathrm{tr}}$ 
& &
$\xdownarrow{0.5cm}{\mathrm{tr}}$ 
\\
Triangle: &$\pol_{n+1}
$&
$\overset{\Grad}{\longrightarrow}$ &
$
\pol_{n}^2
\oplus \b x\times \,\widetilde\pol_n
$
&
$\overset{\Curl}{\longrightarrow}$ &
$
\pol_{n}
$&
\\
&$\xdownarrow{0.5cm}{\mathrm{tr}}$ & &
$\xdownarrow{0.5cm}{\mathrm{tr}}$ \\
Segment: &$\pol_{n+1}
$&
$\overset{\Grad}{\longrightarrow}$ &
$
\pol_{n}
$
\\
\hline
\noalign{\smallskip}
\noalign{\smallskip}
\multicolumn{8}{c}{Second Sequence}\\
\hline
Tetrahedron: &$\pol_{n+2}
$&
$\overset{\Grad}{\longrightarrow}$ &
$
\pol_{n+1}^3
$
&
$\overset{\Curl}{\longrightarrow}$ &
$
\pol_{n}^3
$&
$\overset{\Div}{\longrightarrow}$ &
$\pol_{n-1}$
\\
&$\xdownarrow{0.5cm}{\mathrm{tr}}$ & &
$\xdownarrow{0.5cm}{\mathrm{tr}}$ 
& &
$\xdownarrow{0.5cm}{\mathrm{tr}}$ 
\\
Triangle: &$\pol_{n+2}
$&
$\overset{\Grad}{\longrightarrow}$ &
$
\pol_{n+1}^2
$
&
$\overset{\Curl}{\longrightarrow}$ &
$
\pol_{n}
$&
\\
&$\xdownarrow{0.5cm}{\mathrm{tr}}$ & &
$\xdownarrow{0.5cm}{\mathrm{tr}}$ \\
Segment: &$\pol_{n+2}
$&
$\overset{\Grad}{\longrightarrow}$ &
$
\pol_{n+1},
$
\\
\hline
\end{tabular}
 \label{fig:emb}
\end{figure}

We shall demonstrate explicitly 
show that the above {\it dimensional hierarchy} is satisfied at the level of basis functions for those listed in Theorem \ref{thm:main}.
This concept of dimensional hierarchy was introduced in 
\cite{Zaglmayr06,FuentesKeithDemkowiczNagaraj15}.

To this end, we introduce some further notation.
Consider a triangular face $F = \mathrm{conv}\{\b x_{i_1},\b x_{i_2},\b x_{i_3}\} \in \mathcal{F}(T)$ 
whose outward unit normal vector given as $\b n_F$. 
We denote the set $\mathrm{I}_F :=\{i_1,i_2,i_3\}$ where $i_1, i_2,i_3\in \{1,2,3,4\}$ are three distinct indices, and 
let $i_4$ be the index in $\{1,2,3,4\}\backslash \mathrm{I}_F$.
For any multi-index $\b\alpha=(\alpha_1,\alpha_2,\alpha_3,\alpha_4)\in \INThree{n}$, we denote 
$\b\alpha_F = (\alpha_{i_1},\alpha_{i_2},\alpha_{i_3})\in \INTwo{n}$ to be its restriction onto the face $F$.
We denote  the barycentric coordinates of a point 
$\b x\in F$  with respect to the triangle~$F$ as
$\b\lambda^F=(\lambda_1^F,\lambda_2^F,\lambda_3^F)$ where
\begin{align}
\label{barycentricSum2}
 \b{x}= \sum_{k=1}^{3} \lambda_k^F \b{x}_{i_k};\quad 1 = \sum_{k=1}^{3} \lambda_k^F,
\end{align}
and define the Bernstein polynomials of degree $n\in\ZZ_+$ associated with~$F$ as
\begin{equation}\label{bernsteinPolyDef2}
    B^{F,n}_\b{\alpha}(\b{x}) = {n\choose\b{\alpha}} \b{\lambda^F}(\b x)^\b{\alpha},
    \quad\b{\alpha}\in\INTwo{n},\;\;\;\;\b x\in F.
\end{equation}
The differential operator ``$\Grad$'' acting on scalar function on $F$ is understood to be the 
two-dimensional {\it surface gradient}, and 
``$\Curl$'' acting on a two-dimensional vectorial function on $F$
is understood to be the scalar {\it surface curl}. In other words, if $F$ lies in 
the $x-y$ plane, with $\b n_F = (0,0,1)^T$ as the 
outward unit normal vector, then
\[
 \Grad \phi = \left(\!\!\begin{tabular}{c}
           $\partial_x \phi$\\
           $\partial_y \phi$           
          \end{tabular}
          \!\!\right),\quad  
 \Curl \left(\!\!\begin{tabular}{c}
           $\phi_1$\\
           $\phi_2$           
          \end{tabular}
          \!\!\right) = 
           \partial_x \phi_2-\partial_y \phi_1.          
\]
Denoting the sets 
\begin{subequations}
 \label{2d-sets}
 \begin{align}
  {{\c I}_2^{n}}' :=&\; \INTwo{n}-\{\text{a single, arbitrary index in }\INTwo{n}\},\\
  \check{\c I}_2^{n} :=&\; \INTwo{n}-\{(n,0,0), (0,n,0), (0,0,n)\},
 \end{align}
\end{subequations}
It is known  \cite[Theorem 3.6]{AinsworthAndriamaroDavydov15},
that the following collection of functions forms a basis for $\pol_n(F)^2\oplus \b x\times \widetilde\pol_n(F)$:
\begin{align}
 \label{BB2d-first}
\{ \b\omega_{12}^F,\b\omega_{13}^F,\b\omega_{23}^F\}
\cup 
\{ \Grad{\BB{F,n+1}{\b\alpha}}:\;\b\alpha\in \check{\c I}_2^{n+1}\}
\cup 
\{ \Xa{F,n}{\b\alpha}:\;\b\alpha\in {\INTwo{n'}}\},
\end{align}
where the
{\it lowest-order edge elements}
$\b\omega_{ij}^F$ and the {\it bubble functions}  $\Xa{F,n}{\b\alpha}$
are given by
\begin{align}
 \label{Whitney-2D}
 \b\omega_{ij}^F :=  &\;\lambda_i^F\Grad\lambda_j^F-
 \lambda_j^F\Grad\lambda_i^F,\;\;\;1\le i< j\le 3,\\
 \label{xa}
 \Xa{F,n}{\b\alpha} := &\;(n+1)
 \BB{F,n}{\b\alpha}(\alpha_1\,\b\omega_{23}^F-\alpha_2\,\b\omega_{13}^F+\alpha_3\,\b\omega_{12}^F).
 \end{align}
 Equally well, the collection of functions 
\begin{align}
 \label{BB2d-second}
\{ \b\omega_{12}^F,\b\omega_{13}^F,\b\omega_{23}^F\}
\cup 
\{ \Grad{\BB{F,n+1}{\b\alpha}}:\;\b\alpha\in \check{\c I}_2^{n+1}\}
\cup 
\{ \Xa{F,n-1}{\b\alpha}:\;\b\alpha\in {\INTwo{n-1'}}\},
\end{align}
form a basis for the space $\pol_n(F)^2$.

We define the following trace operators, 
 \begin{alignat}{2}
 \label{traces-defn}
 \trh{F}(\phi) := &\;\phi\restrict{F} && \quad \forall \phi\in H^1(T),\\
 \trt{,F}(
\phi_k\Grad \lambda_{i_k}
 ) :=&\;
 \phi_k\restrict{F}\Grad 
 \lambda_{k}^F
 && \quad \forall \phi_k\in H^1(T),\quad \forall k \in\{ 1,2,3\},\\ 
 \;\;\trn{,F}(\phi_k\Grad \lambda_{i_k}
) := &\;
 \phi_k\restrict{F}\Grad 
 \lambda_{i_k}\cdot \b n_F
 && \quad \forall \phi_k\in H^1(T),\quad \forall k \in\{ 1,2,3\}.
\end{alignat}

It is easy to verify that 
\begin{itemize}
 \item 
 $\trh{F}(\lambda_{i_k}) ={\lambda_k^F}$ for $k \in\{1,2,3\}$,
\item 
 $\trt{,F}(\b\omega_{i_ki_\ell}) =
 \b\omega^F_{k\ell}$ for $k,\ell\in\{1,2,3\}$, 
\item 
 $\trn{,F}(\b\chi_{i_4}) =
 \frac{1}{2\,\mathrm{area}(F)}$,
\end{itemize}
whilst, for $\b\alpha \in \tilde{\c{I}}^{n}_{3}(F)$, we obtain
\begin{itemize}
 \item 
   $\trh{F}(\BB{n}{\b\alpha}) = \BB{F,n}{\b\alpha_F}$, 
   \item
   $\trt{,F}(\Grad\BB{n}{\b\alpha}) = \Grad\BB{F,n}{\b\alpha_F}$, 
  $\trt{,F}(\Xa{FT,n}{\b\alpha}) = \Xa{F,n}{\b\alpha_F}$,
\item 
  $\trn{,F}(\Curl\Xa{FT,n}{\b\alpha}) = \Curl\Xa{F,n}{\b\alpha_F}$ .
\end{itemize}

We can formally state the claim concerning the
dimension hierarchy from a tetrahedron to its triangular face.

\begin{theorem}
\label{thm:dim}
The traces of the three dimensional bases given in Theorem \ref{thm:main} define bases for 
finite element  spaces on triangles:
%
\begin{align*}
  \pol_n(F) &\;= \trh{F}\left(
  \oplus_{k\in\mathrm{I}_F}\{\lambda_k\}\oplus_{E\in \mathcal{E}(F)}\SSS_n^E
 \oplus\SSS_n^F
 \right)
 \\
\pol_{n}(F)^2
\oplus\b x
\times{\widetilde\pol}_{n}(F)
 &\;= \trt{,F}\left(
 \oplus_{k,\ell\in\mathrm{I}_F}\{\b\omega_{k\ell}\}\oplus_{E\in \mathcal{E}(F)}\Grad\SSS_{n+1}^E
 \oplus \Grad\SSS_{n+1}^F\oplus
\EEE_n^F
\right)\\
\pol_{n}(F)^2
 &\;= \trt{,F}\left(
 \oplus_{k,\ell\in\mathrm{I}_F}\{\b\omega_{k\ell}\}\oplus_{E\in \mathcal{E}(F)}\Grad\SSS_{n+1}^E
 \oplus \Grad\SSS_{n+1}^F\oplus
\EEE_{n-1}^F
\right)\\
 \pol_{n}(F) &\;= \trn{,F}\left(
\{\b\chi_{i_4}\}\oplus
\Curl\EEE_{n}^F
 \right)
\end{align*}
Moreover, a similar dimension hierarchy holds for the restriction of basis functions on triangles to its edges.
\end{theorem}

\subsection{Vectorial function spaces}
To end this section, we collect notation for function spaces 
on the tetrahedral~$T$ that will be in the next two sections:
\begin{align*}
\b H(\mathrm{curl}; T):=&\;\{\b v\in L^2(T)^3:\;\; \Curl \b v\in L^2(T)^3\},\\
\b H_0(\mathrm{curl}; T):=&\;\{\b v\in \b H(\mathrm{curl};T):\;\; 
\trt{,F} \b v = \b 0,\;\;
\forall F\in \mathcal{F}(T)
\},\\
\b H(\mathrm{div}; T):=&\;\{\b v\in L^2(T)^3:\;\; \Div \b v\in 
L^2(T)\},\\
\b H_0(\mathrm{div}; T):=&\;\{\b v\in \b H(\mathrm{div};T):\;\; 
\trn{,F} \b v = 0,\;\;
\forall F\in \mathcal{F}(T)
\},\\
\b H(\mathrm{div}^0; T):=&\;\{\b v\in \b H(\mathrm{div};T):\;\; 
\Div \b v = 0\},\\
\b H_0(\mathrm{div}^0; T):=&\;\{\b v\in \b H_0(\mathrm{div};T):\;\; 
\Div \b v = 0\},\\
L^2_0(T):=&\;\{v\in L^2(T):\;\;
\int_T\,v = 0
\}.
\end{align*}


\section{Bernstein-B\'ezier $\b H(\mathrm{curl})$ finite elements on a tetrahedron}
\label{sec:curl}
Here we construct Bernstein-B\'ezier basis for both first- and second-kind \NED 
$\b H(\mathrm{curl})$-finite element on a tetrahedron~$T$:
$
\mathbb{P}_n^3\oplus
\b x\times \widetilde{\mathbb{P}}_n^3
$
and 
$
\mathbb{P}_n^3$.

Following \cite{AinsworthAndriamaroDavydov15,Zaglmayr06},
we seek a basis that
gives a clear separation between the \emph{gradients} of
the polynomial basis for the $H^1$ space, and the non-gradients.
In two dimensions, see \eqref{BB2d-first} and \eqref{BB2d-second}, the non-gradients 
consist of 2D lowest-order edge elements, and 
$\b H(\mathrm{curl})$ face bubbles.
In three dimensions, the non-gradients will 
consist of 3D lowest-order edge elements, the natural 3D 
liftings of the 2D $\b H(\mathrm{curl})$ face bubbles, and a 
set of $\b H(\mathrm{curl})$ cell bubbles, as we see in the next.

\subsection{Lowest-order edge elements}
The following lowest-order edge elements are well-known:
\begin{align}
 \label{ND3D-lowestorder}
 \b\omega_{ij}:=\lambda_i\Grad\lambda_j-
 \lambda_j\Grad\lambda_i,\;\;\;i,j\in\{1,2,3,4\}.
\end{align}
The set
$
 \{
 \b\omega_{ij}:1\le i< j\le 4
 \}
$
form a basis for the space $\NDOne{0}$.

\subsection{Non-gradient $\b H(\mathrm{curl})$ face bubbles}
The non-gradient $\b H(\mathrm{curl})$ face bubbles are obtained by simply lifting the 2D 
$\b H(\mathrm{curl})$ face bubbles on the triangular boundary into the tetrahedron.

Let $F=\mathrm{conv}(\b x_{i_1},\b x_{i_2},\b x_{i_3})$ 
be a face of the tetrahedral~$T$, and let $i_4$ be the index such that 
where $(i_1,i_2,i_3,i_4)$
is  a permutation of $(1,2,3,4)$.
For $\b\alpha\in \tilde{\c I}_{3}^{n}(F)$, 
we define the function
\begin{align}
\label{BB-facebubble}
 \Xa{FT,n}{\b\alpha}:=
(n+1) \BB{n}{\b\alpha}(\alpha_{i_1}\,\b\omega_{i_2i_3}+
\alpha_{i_2}\,\b\omega_{i_3i_1}+
\alpha_{i_3}\,\b\omega_{i_1i_2}).
\end{align}
The following result characterizes
above functions.
\begin{lemma}
\label{lemma:face-curl-bubble}
For any face $F_\ell\in \mathcal{F}(T)$, we have
\begin{align*}
 \trt{,F_\ell}\,\Xa{FT,n}{\b\alpha} =
\left\{
\begin{tabular}{l l}
 $\Xa{F_{\ell},n}{\b\alpha_F} 
 $& if $F_{\ell}=F$,\\
 $0$ & if $F_\ell\not=F$,
\end{tabular}
\right.
\quad \forall \b\alpha\in \tilde{\c{I}}^{n}_{3}(F).
\end{align*}
Moreover,
the set $
 \{\Xa{FT,n}{\b\alpha}:\;
 \b\alpha\in \tilde{\c{I}}^{n}_{3}(F)'
 \}$ 
 consists of ${n+2\choose2}-1$ linearly independent functions
 belonging to $\NDOne{n}$. 
\end{lemma}
\begin{proof}
Since $\b\omega_{ij} \in \NDOne{0}$, we get 
$\Xa{FT,n}{\b\alpha}\in \NDOne{n}$.
The trace property follows from straightforward calculation comparing \eqref{xa} and \eqref{BB-facebubble}. The linear independence of functions is 
a direct consequence of the linear independence of their tangential traces on the face $F$. 
\end{proof}

\subsection{Non-gradient $\b H(\mathrm{curl})$ cell bubbles}
We use the following vector field,
for $\ell\in\{1,2,3,4\}$, to obtain the $\b H(\mathrm{curl})$ cell bubbles:
\begin{align}
\label{BB-cellbubble}
 \Xb{n}{\ell,\b\alpha}
  = (n+1)\BB{n}{\b{\alpha}-\b{e}_\ell}\Grad\lambda_\ell
  - \frac{\alpha_\ell}{n+1}\Grad\BB{n+1}{\b{\alpha}},\;\;
 \b\alpha\in 
 \mathring{\c{I}}_3^{n+1}.
\end{align}

We have the following result 
concerning the above functions,
whose proof is quite technical, and is postponed to the Appendix.
\begin{lemma}
\label{lemma:cell-curl-bubble}
Let $\Xb{n+1}{\ell,\b{\alpha}}$, $\b{\alpha}\in\mathring{\c{I}}^{n+2}_3$, 
denote the vector fields defined in~\eqref{BB-cellbubble}. Then
$\Xb{n+1}{\ell,\b{\alpha}}\in \b{H}_0(\mathrm{curl}; T)\cap\NDOne{n}$. 
Moreover, the set
\begin{align}
\label{BBC-bubble}
       \bigoplus_{\ell=1}^2\left\{
      \Xb{n+1}{\ell,\b{\alpha}}: \b{\alpha}\in
       \mathring{\c{I}}_3^{n+2}
    \right\}
    \bigoplus\left\{    
      \Xb{n+1}{3,\b{\alpha}}: \b{\alpha}\in  \mathring{\c{I}}_3^{n+2}\;,\alpha_3=1
    \right\} 
\end{align}
consists of $2{n+1\choose{3}} +{n\choose{2}}$ linearly independent 
vector fields in $\b{H}_0(\mathrm{curl}; T)\cap\NDOne{n}$, 
whose curl set 
\begin{align}
 \label{BBC-bubble-curl}
       \bigoplus_{\ell=1}^2\left\{
      \Curl\Xb{n+1}{\ell,\b{\alpha}}: \b{\alpha}\in
       \mathring{\c{I}}_3^{n+2}
    \right\}
    \bigoplus\left\{    
      \Curl\Xb{n+1}{3,\b{\alpha}}: \b{\alpha}\in  \mathring{\c{I}}_3^{n+2}\;,\alpha_3=1
    \right\} 
\end{align}
form a basis for the divergence-free bubble space
$\b H_0(\mathrm{div}^0,T)\cap \mathbb{P}_{n}^3$.
\end{lemma}

\subsection{Bernstein-B\'ezier $\b H(\mathrm{curl})$ basis}
Combing the results in Lemma \ref{lemma:face-curl-bubble} and Lemma \ref{lemma:cell-curl-bubble},
and adding back the lowest-order edge elements along with the gradient fields, 
we obtain the Bernstein-B\'ezier $\b H(\mathrm{curl})$ basis for both $\NDOne{n} (n\ge 0)$, and $\NDTwo{n} (n\ge 1)$
in the following theorem.

\begin{theorem}
\label{thm:curl}
The set
\begin{subequations}
\label{BB-curl-basis}
\begin{alignat}{2}
 \label{BB-curl-basis-1}
&
\bigoplus
\{\b\omega_{ij}:\;1\le i<j\le 4\}&&\;\text{(lowest order)}\\
&
\;
\bigoplus
\{\Grad \BB{n+1}{\b\alpha}:\;\b\alpha\in \check{\c{I}}_{3}^{n+1}\}&&\;
\text{(gradient fields)}\nonumber\\
&
\;\;
\bigoplus_{F\in \mathcal{F}(T)}
\{
 \Xa{FT,n}{\b\alpha}:\;
 \b\alpha\in \tilde{\c{I}}^{n}_{3}(F)'
 \}&&\;\text{(face bubbles)}
 \nonumber\\
 &
\;\;\;
\bigoplus_{\ell=1}^2\left\{
      \Xb{n+1}{\ell,\b{\alpha}}: \b{\alpha}\in
       \mathring{\c{I}}_3^{n+2}
    \right\}
    \bigoplus\left\{    
      \Xb{n+1}{3,\b{\alpha}}: \b{\alpha}\in  \mathring{\c{I}}_3^{n+2}\;,\alpha_3=1
    \right\} 
    &&\;\text{(cell bubbles)}
    \nonumber
\end{alignat}
forms a basis for $\NDOne{n}$,
and the set
\begin{alignat}{2}
\label{BB-curl-basis-2}
&
\bigoplus
\{\b\omega_{ij}:\;1\le i<j\le 4\}&&\;\text{(lowest order)}\\
&
\;
\bigoplus
\{\Grad \BB{n+1}{\b\alpha}:\;\b\alpha\in \check{\c{I}}_{3}^{n+1}\}&&\;
\text{(gradient fields)}\nonumber\\
&
\;\;
\bigoplus_{F\in\mathcal{F}(T)}
\{
 \Xa{FT,n-1}{\b\alpha}:\;
 \b\alpha\in \tilde{\c{I}}^{n-1}_{3}(F)'
 \}&&\;\text{(face bubbles)}
 \nonumber\\
 &
\;\;\;
\bigoplus_{\ell=1}^2\left\{
      \Xb{n}{\ell,\b{\alpha}}: \b{\alpha}\in
       \mathring{\c{I}}_3^{n+1}
    \right\}
    \bigoplus\left\{    
      \Xb{n}{3,\b{\alpha}}: \b{\alpha}\in  \mathring{\c{I}}_3^{n+1}\;,\alpha_3=1
    \right\} 
    &&\;\text{(cell bubbles)}\nonumber
\end{alignat}
\end{subequations}
forms a basis for $\NDTwo{n}$.
\end{theorem}
\begin{proof}
Let us focus on the proof of the first statement.
The proof for the second is identical
to that for the first, and is omitted.

By Lemma~\ref{lemma:face-curl-bubble} and Lemma~\ref{lemma:cell-curl-bubble}, 
we have functions contained in the respective sets in \eqref{BB-curl-basis-1} lie in 
$\NDOne{n}$, whose dimensions are $6$,
${n+4\choose 3}-4$, $4 {n+2\choose 2}-4$, and 
$2{n+1\choose{3}} +{n\choose{2}}$, leading to 
$\frac{1}{2}(n+1)(n+3)(n+4)=\dim\NDOne{n}$ functions in total.

We are left to show the linear independence of functions in these sets.
Concerning the tangential trace on a face $F_1=\mathrm{conv}\{\b x_{2},\b x_{i_3},\b x_{i_4}\}$ of functions in each set, we have 
that only the following functions have non-vanishing traces:
\begin{alignat*}{2}
&\bigoplus
\{\b\omega_{k\ell}:\;k,\ell\in \{2,3,4\}\}&&\;\text{(lowest order)}\\
&
\;
\bigoplus
\{\Grad \BB{n+1}{\b\alpha}:\;\b\alpha\in \check{\c{I}}_{3}^{n+1}, \b\alpha_{1}=0\}&&\;
\text{(gradient fields)}\\
&
\;\;
\bigoplus
\{
 \Xa{F_4T,n}{\b\alpha}:\;
 \b\alpha\in \tilde{\c{I}}^{n}_{3}(F_4)'
 \}&&\;\text{(face bubbles)}.
\end{alignat*}
The tangential traces of the above functions, in turn,
form a set of basis for $\NDOne{n}$ on face $F_4$; see Theorem \ref{thm:dim}.
Hence, the above functions are themselves linearly independent.
Working 
on the other faces, we 
obtain the linear independence of the functions in \eqref{BB-curl-basis-1} with non-vanishing tangential traces.

The remaining tangential-trace-free functions are the gradient bubbles 
\[
\{\Grad \BB{n+1}{\b\alpha}:\;\b\alpha\in \mathring{\c{I}}_{3}^{n+1}\}
\]
and the non-gradient bubbles in the set \eqref{BBC-bubble}.
The linear independence of these two sets is a direct consequence of Lemma
\ref{lemma:cell-curl-bubble}.
This completes the proof.
\end{proof}

\section{Bernstein-B\'ezier $\b H(\mathrm{div})$ finite element on a tetrahedron}
\label{sec:div}
Here we construct Bernstein-B\'ezier basis for both Raviart-Thomas (RT) and Brezzi-Douglas-Marini(BDM) $\b H(\mathrm{div})$-finite
element on a tetrahedral~$T$:
$\RT{n}$ and $\mathbb{P}_n^3$.

Again, we seek a basis that
gives a clear separation between the 
\emph{curl} of the Bernstein polynomial basis for the $\b H(\mathrm{curl})$ space
constructed in the previous section, 
and the non-curl (not divergence-free) 
$\b H(\mathrm{div})$ bubbles.

\subsection{Lowest-order face elements}
The following 
{\it lowest-order face elements} are well-known:
\begin{align}
 \label{DIV3D-lowestorder}
 \b\chi_{\ell}:=\lambda_i\Grad\!\lambda_j\!\times\!\Grad\!\lambda_k-
 \lambda_j\Grad\!\lambda_i\!\times\!\Grad\!\lambda_k+
 \lambda_k\Grad\!\lambda_i\!\times\!\Grad\!\lambda_j,\;1\le i\!< \!j<\!k \le 4,
\end{align}
where $(i,j,k,\ell)$ is a permutation of $(1,2,3,4)$.
The set
$
 \{
 \b\chi_{\ell}:1\le \ell \le 4
 \}
$
form a basis of the space $\RT{0}$.

\subsection{Non-curl $\b H(\mathrm{div})$ bubbles}
The non-curl $\b H(\mathrm{div})$ bubbles take a similar form as their two-dimensional companions 
in \cite[Theorem 3.2]{AinsworthAndriamaroDavydov15}.
We use the following vector field to obtain the $\b H(\mathrm{div})$ bubbles:
\begin{align}
\label{BB-divbubble}
 \Xc{n}{\b\alpha}
  = (n+1)\BB{n}{\b{\alpha}}\sum_{\ell=1}^4(-1)^\ell\alpha_\ell\b\chi_{\ell},\;\;
 \b\alpha\in 
\INThree{n}.
\end{align}

The following result characterizes the above functions,
whose proof follows 
that for the two-dimensional case in \cite[Lemma 3.1]{AinsworthAndriamaroDavydov15}.
\begin{lemma}
\label{lemma:div-bubble}
We have $\Xc{n}{\b{\alpha}}\in \b{H}_0(\mathrm{div}; T)\cap\RT{n}$. 
Moreover, the set
\begin{align}
\label{BBD-bubble}
\left\{
      \Xc{n}{\b{\alpha}}: \b{\alpha}\in
       \INThree{n'}
    \right\}
\end{align}
consists of ${n+3\choose{3}} -1$ linearly independent 
vector fields in $\b{H}_0(\mathrm{div}; T)\cap\RT{n}$, 
whose divergence set 
\begin{align}
\label{BBD-bubble-div}
\left\{
      \Div\Xc{n}{\b{\alpha}}: \b{\alpha}\in
       \INThree{n'}
    \right\}
\end{align}
form a basis for the {\it average-zero} polynomial functions
$L^2_0(T)\cap \mathbb{P}_{n}$.
\end{lemma}
\begin{proof}
Since $\b\chi_\ell\in\RT{0}$ whose normal trace is constant on a single face and vanishes elsewhere,
it is trivial to show  
$\Xc{n}{\b{\alpha}}\in \b{H}_0(\mathrm{div}; T)\cap\RT{n}$ and 
$\Div\Xc{n}{\b{\alpha}}\in {L}^2_0(T)\cap\mathbb{P}_{n}$.
Since the number of functions in either of the sets \eqref{BBD-bubble} 
and \eqref{BBD-bubble-div}
is $\#\INThree{n'}={n+3\choose3}-1$, which equals the dimension of the space
$L^2_0(T)\cap \mathbb{P}_{n}$, we only need to prove
the functions in the set 
\eqref{BBD-bubble-div} are linearly independent to conclude the proof of Lemma \ref{lemma:div-bubble}. 
We follow the proof of \cite[Lemma 3.1]{AinsworthAndriamaroDavydov15}.

Let coefficients $\cd$ be such that 
\[
 \sum_{\b\alpha\in \INThree{n}}\cd\Div\Xc{n}{\b\alpha} = 0.
\]
To simplify notation, we denote 
\[
 \epsilon_{ijk} := \Grad\lambda_i\cdot(\Grad\lambda_j\times  \Grad\lambda_k),
\]
so that, in particular,
\[
  \epsilon_{123}= -\epsilon_{124}= \epsilon_{134}= -\epsilon_{234}.
\]

Given 
a permutation $(i,j,k,\ell)$ of $(1,2,3,4)$ with $i<j<k$,
a simple calculation reveals that 
\begin{align}
\label{div-0}
 \Div\left(\BB{n}{\b\alpha}\,\lambda_i\Grad\lambda_j\times \Grad\lambda_k\right)
 = (\alpha_i+1)(\BB{n}{\b\alpha}-
 \BB{n}{\b\alpha+\b e_i-\b e_\ell})\epsilon_{ijk}.
\end{align}
By definition \eqref{BB-divbubble}, we notice that
$\Xc{n}{\b{\alpha}}$ consists of a linear combination of 
$4\times 3 =12$ terms of the form 
$\BB{n}{\b\alpha}\,\lambda_i\Grad\lambda_j\times \Grad\lambda_k$.
Applying \eqref{div-0}, and after straightforward algebraic manipulation, we obtain that 
\[
  \sum_{\b\alpha\in \INThree{n}}\cd\Div\Xc{n}{\b\alpha} = 
   \epsilon_{123}(n+1)\sum_{\b\alpha\in \INThree{n}}\cdt\BB{n}{\b\alpha},
\]
where 
\begin{align}
 \cdt := \sum_{i=i}^4\sum_{\overset{j=1}{j\not=i}}^4
 \alpha_i(\alpha_j+1)
 \,\cd
 -
 \sum_{i=i}^4\sum_{\overset{j=1}{j\not=i}}^4
 (\alpha_i+1)\alpha_j\,\cdn.
\end{align}
The linear independence of the Bernstein polynomials means that 
$\cdt=0$ for all $\b\alpha\in\INThree{n}$. This latter
condition constitutes a square linear system for the coefficients
$\cd$. The system is irreducible and weakly diagonally dominant, with
the coefficients in each row of the system summing to zero. 
Hence, there exists $c\in\RR$ such that $\cd=c$ 
for all $\b{\alpha}\in\INThree{n}$, which immediately implies the linear independence of functions in the set \eqref{BBD-bubble-div}. This completes the proof.
\end{proof}

\subsection{Bernstein-B\'ezier $\b H(\mathrm{div})$ basis}
Combing the results in Lemma \ref{lemma:face-curl-bubble}, Lemma \ref{lemma:cell-curl-bubble} 
and Lemma \ref{lemma:div-bubble}, 
we are ready to present the Bernstein-B\'ezier $\b H(\mathrm{div})$ basis for both $\RT{n} (n\ge 0)$, and $\pol_{n}^3 (n\ge 1)$.
The results are collected in the following theorem.

\begin{theorem}
\label{thm:div}
The set
\begin{subequations}
\label{BB-div-basis}
\begin{alignat}{2}
 \label{BB-div-basis-1}
\bigoplus_{\ell=1}^4
\{\b\chi_{\ell}\}\;
\bigoplus \b W^n\;
\bigoplus\{ 
      \Xc{n}{\b{\alpha}}: \b{\alpha}\in
       \INThree{n'}
\}
 \end{alignat}
forms a basis for $\RT{n}$,
and the set
\begin{alignat}{2}
 \label{BB-div-basis-2}
\bigoplus_{\ell=1}^4
\{\b\chi_{\ell}\}\;
\bigoplus \b W^n\;
\bigoplus\{ 
      \Xc{n-1}{\b{\alpha}}: \b{\alpha}\in
       \INThree{n-1'}
\}
 \end{alignat}
\end{subequations}
forms a basis for $\pol_{n}^3$.
Here the set $\b W^n$ contains the divergence-free fields obtained by taking the curl of 
the $H(\mathrm{curl})$-basis functions:
\begin{align*}
 \b W^n = &\;
 \bigoplus_{F\in \mathcal{F}(T)}
\{
\Curl \Xa{FT,n}{\b\alpha}:\;
 \b\alpha\in \tilde{\c{I}}^{n}_{3}(F)'
 \}\\
&\;\;\;\;\;\bigoplus_{\ell=1}^2\left\{
      \Curl\Xb{n+1}{\ell,\b{\alpha}}: \b{\alpha}\in
       \mathring{\c{I}}_3^{n+2}
    \right\}
\bigoplus\left\{ \Curl   
      \Xb{n+1}{3,\b{\alpha}}: \b{\alpha}\in  \mathring{\c{I}}_3^{n+2}\;,\alpha_3=1
    \right\}.
\end{align*}
\end{theorem}
\begin{proof}
The fact that the functions in the sets lie in the finite element spaces and are linearly independent are due to 
Lemma \ref{lemma:face-curl-bubble}, Lemma \ref{lemma:cell-curl-bubble}, 
and Lemma \ref{lemma:div-bubble}.

The number of functions contained in the respective sets in \eqref{BB-div-basis-1} are
$4$, $4{n+2\choose2}-4+2{n+1\choose3}+{n\choose2}$, ${n+3\choose 3}-1$, leading to $\frac{1}{2}(n+1)(n+2)(n+4)=\dim \RT{n}$ functions in total; 
similarly, the total number of functions in \eqref{BB-div-basis-2} is
\begin{align*}
  4+4{n+2\choose2}-4+2{n+1\choose3}+{n\choose2}+
 {n+2\choose 3}-1 = &\;
 \frac{1}{2}(n+1)(n+2)(n+3)\\
 =&\;\dim \pol_n^3.
\end{align*}
Hence the sets \eqref{BB-div-basis-1} and \eqref{BB-div-basis-2} form 
bases for $\RT{n}$ and $\pol_n^3$, respectively.
This completes the proof.
\end{proof}

\section{Efficient computation of mass and stiffness matrices with the basis}
\label{sec:comp}
In this section, we show how to explicitly compute the mass and stiffness matrices associated with the 
$\b H(\mathrm{curl})$- and $\b H(\mathrm{div})$-basis functions developed in the previous two sections
for piecewise constant coefficients on an affine tetrahedron. 
Such explicit computation is made possible by the fact that our basis functions can be expressed as a linear combination of Bernstein polynomials, and then exploiting the fact that
the product of two Bernstein polynomials is another scaled Bernstein polynomial, and the integral of Bernstein polynomial has explicit formula.

\newcommand{\Szero}[1]{\mathrm{S_{#1}^0}}
\newcommand{\Sone}[1]{\mathrm{S_{#1}^1}}
\newcommand{\Stwo}[1]{\mathrm{S_{#1}^2}}
\newcommand{\Cab}[1]{\mathrm{C_{#1}}}
\newcommand{\Mab}[1]{\mathrm{M_{#1}}}

\subsection{Notation}
The following geometric quantities of the tetrahedron $T$ will be used:
\begin{subequations}
\begin{align}
 \label{geom-1}
     \b t_{ij} = &\;\Grad \lambda_i\times\Grad\lambda_j,\\
  \epsilon_{ijk} =&\;
  \b t_{ij}\cdot \lambda_k,
\end{align}
and
 \label{geom}
 \begin{align}
  \Szero{i,j} = &\;\int_T \Grad \lambda_i\cdot\Grad\lambda_j\, \mathrm{dx},\\
  \Sone{ij,kl} =&\; \int_T 
  \b t_{ij}\cdot  \b t_{kl}
  \, \mathrm{dx},\\
\Stwo{ijk} =&\; \int_T 
 (\epsilon_{ijk})^2\, \mathrm{dx}.
 \end{align}
\end{subequations}
We also use the Bernstein coefficients, which is independent of geometry:
 \begin{align}
 \label{bb-coef}
   \Mab{\b\alpha,\b\beta}= &\frac{
{\b\alpha+\b\beta\choose\b\alpha}}
{  {|\b\alpha|+|\b\beta|\choose|\b\alpha|}\;
    {|\b\alpha|+|\b\beta|+3\choose3}}
 \end{align}
 By properties of Bernstein polynomials, we have
 \[
  \BB{n}{\b{\alpha}}
    \BB{m}{\b{\beta}} =
    \frac{
{\b\alpha+\b\beta\choose\b\alpha}}
{  {n+m\choose n}}
    \BB{n+m}{\b{\alpha}+\b\beta},
    \quad\quad
\int_T    \BB{n}{\b{\alpha}} \BB{m}{\b{\beta}}\,\mathrm{dx} = 
\mathrm{M_{\b\alpha,\b\beta}}|T|.
 \]

\subsection{Transform to pure Bernstein-B\'ezier form}
First, let us transform all the basis functions along with their differential operators
back to a linear combination of Bernstein polynomials. 

\subsubsection{$\b H(\mathrm{curl})$-basis transformation}
We have four type of basis functions, namely, the lowest order edge elements $\b\omega_{ij}$ (Type 1), the gradient fields 
$\Grad\BB{n+1}{\b\alpha}
$ (Type 2), 
the face  bubbles $\Xa{FT,n}{\b\alpha}$ (Type 3), and the cell bubbles
$\Xb{T,n+1}{\ell,\b\alpha}$ (Type 4).

For compactness of the presentation, 
we just derive basis transformation for Type 3 basis below, 
and leave out details for the others.
The transformation of these basis functions back to Bernstein-B\'ezier forms, along with their curls, are listed in Table \ref{table-curl}, with the corresponding coefficients given in Table \ref{table-curl-coef}.

For Type 3 basis for the face $F_\ell = \mathrm{conv}\{\b x_i,\b x_j,\b x_k\}$, we have 
\begin{align*}
 \Xa{F_{\ell}T,n}{\b\alpha}= &\; 
(n+1) \BB{n}{\b\alpha}(\alpha_{i}\,\b\omega_{jk}-
 \alpha_{j}\,\b\omega_{ik}+\alpha_{k}\,\b\omega_{ij})\\
 = &\; 
(n+1) \BB{n}{\b\alpha}\Big(\lambda_i(\alpha_k\Grad\lambda_j-\alpha_j\Grad\lambda_k)
\\
&\;
\hspace{1.6cm}+
\lambda_j(\alpha_i\Grad\lambda_k-\alpha_k\Grad\lambda_i)
\\
&\;\hspace{1.6cm}
+
\lambda_k(\alpha_j\Grad\lambda_i-\alpha_i\Grad\lambda_j)
\Big)\\
 = &\; 
\BB{n+1}{\b\alpha+\b e_i}(\alpha_i+1)(\alpha_k\Grad\lambda_j-\alpha_j\Grad\lambda_k)
\\
&\;
+
\BB{n+1}{\b\alpha+\b e_j}(\alpha_j+1)(\alpha_i\Grad\lambda_k-\alpha_k\Grad\lambda_i)
\\
&\;
+
\BB{n+1}{\b\alpha+\b e_k}(\alpha_k+1)(\alpha_j\Grad\lambda_i-\alpha_i\Grad\lambda_j)
\\
=&\;
\sum_{\b\sigma\in \Sigma_\ell}
\BB{n+1}{\b\alpha+\b e_{\sigma_1}}({\alpha_{\sigma_1}+1})
(\alpha_{\sigma_3}\Grad\lambda_{\sigma_2}-
\alpha_{\sigma_2}\Grad\lambda_{\sigma_3})
\end{align*}
 where $\b\sigma = (\sigma_1,\sigma_2,\sigma_3)$ and the set 
 \begin{align}
  \label{set-f}
    \Sigma_\ell:=\{(i,j,k), (j,k,i), (k,i,j)\},
 \end{align}
where
 $(i,j,k,\ell)$ form a permutation of $(1,2,3,4)$.
We also have 
\begin{align*}
\Curl \Xa{F_{\ell}T,n}{\b\alpha}
=&\;
\sum_{\b\sigma\in \Sigma_\ell}
({\alpha_{\sigma_1}+1})\Grad\BB{n+1}{\b\alpha+\b e_{\sigma_1}}\times
(\alpha_{\sigma_3}\Grad\lambda_{\sigma_2}-
\alpha_{\sigma_2}\Grad\lambda_{\sigma_3})\\
=&\;
(n+1)\sum_{\b\sigma\in \Sigma_\ell}\sum_{k=1}^3
\BB{n}{\b\alpha+\b e_{\sigma_1}\!\!-\b e_{\sigma_k}}
({\alpha_{\sigma_1}+1})
(\alpha_{\sigma_3}\b t_{\sigma_k\sigma_2}-
\alpha_{\sigma_2}\b t_{\sigma_k\sigma_3})
\end{align*}

\begin{table}[ht]
\caption{Bernstein-B\'ezier transformation of $\b H(\mathrm{curl})$-basis functions}
 \centering
\begin{tabular}{l| c| c |c }
\hline
\noalign{\smallskip}
      & basis &  BB-form & BB-form of its curl \\
      \hline
\noalign{\smallskip}
Type 1& $\b\omega_{ij}$ &
$\BB{1}{\b e_i}\Grad\lambda_j-\BB{1}{\b e_j}\Grad\lambda_i$
&$2\,\b t_{ij}$\\
\noalign{\smallskip}
\hline
\noalign{\smallskip}
Type 2& $\Grad\BB{n+1}{\b\alpha}$ &
$ \sum_{i = 1}^4  \b{{c}}_{\b\alpha,i}^{(2)}\,\,\BB{n}{\b\alpha-\b e_i}$
&$\b 0$\\
\noalign{\smallskip}
\hline
\noalign{\smallskip}
Type 3& 
$\Xa{F_{\ell}T,n}{\b\alpha}$&
$\sum_{\b\sigma\in \Sigma_\ell}
 {\b{c}}_{\b\alpha,\b\sigma}^{(3)} \,\,
\BB{n+1}{\b\alpha+\b e_{\sigma_1}}
$
&
$\sum_{\b\sigma\in \Sigma_\ell}
\sum_{k=1}^3
\b d_{\b\alpha,\b\sigma,k}^{(2)}\;
\BB{n}{\b\alpha+\b e_{\sigma_1}-\b e_{\sigma_k}}
$\\
\noalign{\smallskip}
\hline
\noalign{\smallskip}
Type 4& $
\Xb{T,n+1}{\ell,\b\alpha}$&
$  \sum_{i=1}^4 \b c_{\b\alpha,\ell,i}^{(4)}\;\BB{n+1}{\b{\alpha}-\b e_i}$&
$  \sum_{i=1}^4 \b d_{\b\alpha,\ell,i}^{(3)}\;\BB{n}{\b{\alpha}-\b e_\ell-\b e_i}$
\\
\noalign{\smallskip}
\hline
\end{tabular}
\label{table-curl}
\end{table}

\begin{table}[ht]
\caption{Bernstein-B\'ezier transformation of $\b H(\mathrm{curl})$-basis functions}
 \centering
\begin{tabular}{r l}
\hline
\noalign{\smallskip}
${\b{c}}_{\b\alpha,i}^{(2)}$\!\!\! &$=\;
 (n+1)\Grad\lambda_i$
 \\
\noalign{\smallskip}
\hline
\noalign{\smallskip}
$ {\b{c}}_{\b\alpha,\b\sigma}^{(3)}$\!\!\! &$=\;
({\alpha_{\sigma_1}+1})
(\alpha_{\sigma_3}\Grad\lambda_{\sigma_2}-
\alpha_{\sigma_2}\Grad\lambda_{\sigma_3})$
\\
\noalign{\smallskip}
\hline
\noalign{\smallskip}
$ {\b{d}}_{\b\alpha,\b\sigma,k}^{(2)} $\!\!\!&$=\;
(n+1)({\alpha_{\sigma_1}+1})
(\alpha_{\sigma_3}\b t_{\sigma_k\sigma_2}-
\alpha_{\sigma_2}\b t_{\sigma_k\sigma_3})$
\\
\noalign{\smallskip}
\hline
\noalign{\smallskip}
$ {\b{c}}_{\b\alpha,\ell, i}^{(4)} $\!\!\!&$=\;
 (\delta_{il}(n+2)-\alpha_i)\Grad\lambda_i$\\
\noalign{\smallskip}
\hline
\noalign{\smallskip}
$ {\b{d}}_{\b\alpha,\ell,i}^{(3)} $\!\!\!&$=\;
(n+2)(n+1)\b t_{i\ell}$\\
\noalign{\smallskip}
\hline
 \end{tabular}
 \label{table-curl-coef}
\end{table}
Here $\delta_{i\ell}$ is the Kronecker delta.

\subsubsection{$\b H(\mathrm{div})$-basis transformation}
We also have four type of basis functions, namely, the lowest order face elements $\b\chi_{\ell}$ (Type 1), the div-free face bubbles 
$\Curl\Xa{F_{\ell}T,n}{\b\alpha}
$ (Type 2), 
the div-free cell  bubbles $\Curl\Xb{T,n+1}{\ell,\b\alpha}
$ (Type 3), and the non-curl cell bubbles
$ \Xc{n}{\b\alpha}
$ (Type 4).
The transformation of these 
basis functions back to Bernstein-B\'ezier forms, along with their divergence, are listed in Table \ref{table-div}, 
where the coefficients $\b d_{\b\alpha,\b\sigma,k}^{(2)}$ and 
$\b d_{\b\alpha,\ell, i}^{(3)}$ are given in Table \ref{table-curl-coef}, and 
the coefficients $\b d_{\b\alpha,i}^{(4)}$ and 
$\b g_{\b\alpha, i,m}^{(2)}$ are given in Table \ref{table-div-coef}. Again, we leave out details of the derivation.
Here $(i,j,k,l)$ form a cyclic permutation of $(1,2,3,4)$.
\begin{table}[ht!]
\caption{Bernstein-B\'ezier transformation of $\b H(\mathrm{div})$-basis functions}
 \centering
\begin{tabular}{l| c| c |c }
\hline
\noalign{\smallskip}
      & basis &  BB-form & BB-form of its divergence \\
      \hline
\noalign{\smallskip}
Type 1& $\b\chi_{\ell}$ &
$
\sum_{\b\sigma\in\Sigma_{\ell^p}}\b t_{\sigma_2\sigma_3}\,
\BB{1}{\b e_{\sigma_1}}
$
&$(-1)^\ell3\,\epsilon_{123}$\\
\noalign{\smallskip}
\hline
\noalign{\smallskip}
Type 2& $\Curl\Xa{F_{\ell}T,n}{\b\alpha}$&$
\sum_{\b\sigma\in \Sigma_\ell}
\sum_{k=1}^4
\b d_{\b\alpha,\b\sigma,k}^{(2)}\;
\BB{n}{\b\alpha+\b e_{\sigma_1}-\b e_k}
$ &
$ 0$\\
\noalign{\smallskip}
\hline
\noalign{\smallskip}
Type 3& $\Curl\Xb{T,n+1}{\ell,\b\alpha}$&
$  \sum_{i=1}^4 \b d_{\b\alpha,\ell,i}^{(3)}\;\BB{n}{\b{\alpha}-\b e_\ell-\b e_i}$
&
$0
$\\
\noalign{\smallskip}
\hline
\noalign{\smallskip}
Type 4& $
   \Xc{n}{\b\alpha}
  $ &$
\sum_{i=1}^4\b d_{\b\alpha,i}^{(4)}\BB{n+1}{\b{\alpha}+\b e_i}
  $&
$\sum_{i=1}^4\sum_{j=1}^4 g_{\b\alpha, i,j}^{(2)}\;\;
  \BB{n}{\b{\alpha}+\b e_i-\b e_j}
$
\\
\noalign{\smallskip}
\hline
\end{tabular}
\label{table-div}
\end{table}
\begin{table}[ht!]
\caption{The coefficients}
 \centering
\begin{tabular}{r l}
\hline
\noalign{\smallskip}
$\b d_{\b\alpha,i}^{(4)}$\!\!\! &$=\;
 (-1)^{i+1}(\alpha_i+1)\sum_{\b\sigma\in\Sigma_i}\alpha_{\sigma_1}\,\b t_{\sigma_2\sigma_3}
  $
 \\
\noalign{\smallskip}
\hline
\noalign{\smallskip}
$ g_{\b\alpha, i,j}^{(2)}$\!\!\! &$=\;
(n+1)(\alpha_i+1)(\delta_{ij}\,n-\alpha_j)\, \epsilon_{123}
$
\\
\noalign{\smallskip}
\hline
 \end{tabular}
 \label{table-div-coef}
\end{table}

\subsection{Mass and stiffness matrices for $\b H(\mathrm{curl})$ basis functions}
Now, let us give formula for the mass and stiffness matrices for the $\b H(\mathrm{curl})$ basis functions.
For ease of notation, suppose that we have ordered the basis functions such that
\begin{align}
 \label{curl-ordering}
 \b\phi_p^1 = \b\omega_{i^pj^p},\;\;
 \b\phi_p^2 = \Grad\BB{n+1}{\b\alpha^p},\;\;
 \b\phi_p^3 = \Xa{F_{\ell^p}T,n}{\b\alpha^p},\;\;
 \b\phi_p^4 = \Xb{T,n+1}{\ell^p,\b\alpha^p},
\end{align}
where $p\in \mathbb{N}^+$ is an integer and the superscript $p$ above indicates the corresponding indices are determined by $p$.

Due to the heterogeneity of the basis functions, the mass and stiffness matrices consists of $4\times 4$ blocks. 
By symmetry, we need to work out the computation for the $10$ upper diagonal blocks for each matrix. 
Our calculation is given below.

\subsubsection{$\b H(\mathrm{curl})$ mass matrix}
The first four blocks:
\begin{subequations}
 \label{curl-mass}
 \begin{align}
\int_T  \b\phi_p^1\cdot\b\phi_q^1\,\mathrm{dx} = &\;
\Mab{\b e_{i^p},\b e_{i^q}}\Szero{j^p,j^q}-
\Mab{\b e_{i^p},\b e_{j^q}}\Szero{j^p,i^q}\nonumber\\
&\;-
\Mab{\b e_{j^p},\b e_{i^q}}\Szero{i^p,j^q}+
\Mab{\b e_{j^p},\b e_{j^q}}\Szero{i^p,i^q}\\
\int_T  \b\phi_p^1\cdot\b\phi_q^2\,\mathrm{dx} = &\;
(n+1)\sum_{s=1}^4 \left(\Mab{\b e_{i^p},\b \alpha^q-\b e_s}
\Szero{j^p,s}-\Mab{\b e_{j^p},\b \alpha^q-\b e_s}
\Szero{i^p,s}\right)\\
\int_T  \b\phi_p^1\cdot\b\phi_q^3\,\mathrm{dx} = &\;
\!\!\!\sum_{\b\sigma\in\Sigma_{\ell^q}}\!\!\! \Mab{\b e_{i^p},\b\alpha^q+\b e_{\sigma_1}}
({\alpha^q_{\sigma_1}+1})
(\alpha^q_{\sigma_3}\Szero{j^p,\sigma_2}-
\alpha^q_{\sigma_2}\Szero{j^p,\sigma_3})\nonumber\\
&-\sum_{\b\sigma\in\Sigma_{\ell^q}}\!\!\! \Mab{\b e_{j^p},\b\alpha^q+\b e_{\sigma_1}}
({\alpha^q_{\sigma_1}+1})
(\alpha^q_{\sigma_3}\Szero{i^p,\sigma_2}-
\alpha^q_{\sigma_2}\Szero{i^p,\sigma_3})\\
\int_T  \b\phi_p^1\cdot\b\phi_q^4\,\mathrm{dx} = &\;
\sum_{s=1}^4 (\delta_{s\ell^q}(n+2)-\alpha^q_s)(\Mab{\b e_{i^p},\b \alpha^q-\b e_s}
\Szero{j^p,s}-
\Mab{\b e_{j^p},\b \alpha^q-\b e_s}
\Szero{i^p,s})
 \end{align}
The next three blocks:
 \begin{align}
\int_T  \b\phi_p^2\cdot\b\phi_q^2\,\mathrm{dx} = &\;
(n+1)^2\sum_{t=1}^4\sum_{s=1}^4 
\Mab{\b \alpha^p-\b e_t,\b \alpha^q-\b e_s}\,
\Szero{t,s}\\
\int_T  \b\phi_p^2\cdot\b\phi_q^3\,\mathrm{dx} = &\;
(n+1)\sum_{t=1}^4\sum_{\b\sigma\in\Sigma_{\ell^q}}
\!\!\! \Mab{\b\alpha^p-\b e_t,\b\alpha^q+\b e_{\sigma_1}}({\alpha^q_{\sigma_1}+1})
(\alpha^q_{\sigma_3}\Szero{t,\sigma_2}\!\!-
\alpha^q_{\sigma_2}\Szero{t,\sigma_3})\\
\int_T  \b\phi_p^2\cdot\b\phi_q^4\,\mathrm{dx} = &\;
(n+1)\sum_{t=1}^4\sum_{s=1}^4 (\delta_{s\ell^q}(n+2)-\alpha^q_s)\Mab{\b\alpha^p-\b e_t,\b \alpha^q-\b e_s}
\Szero{t,s}
 \end{align}
The next two blocks:
 \begin{align}
\int_T  \b\phi_p^3\cdot\b\phi_q^3\,\mathrm{dx} = &\;
\sum_{\b\tau\in\Sigma_{\ell^p}}\sum_{\b\sigma\in\Sigma_{\ell^q}}
\!\!\! \Mab{\b\alpha^p+\b e_{\tau_1},\b\alpha^q+\b e_{\sigma_1}}
({\alpha^p_{\tau_1}+1})
({\alpha^q_{\sigma_1}+1})\nonumber\\
&\;(\alpha^p_{\tau_3}\alpha^q_{\sigma_3}\Szero{\tau_2,\sigma_2}\!\!-
\alpha^p_{\tau_3}\alpha^q_{\sigma_2}\Szero{\tau_2,\sigma_3}
\!\!\!-
\alpha^p_{\tau_2}\alpha^q_{\sigma_3}\Szero{\tau_3,\sigma_2}
\!\!\!+\alpha^p_{\tau_2}\alpha^q_{\sigma_2}\Szero{\tau_3,\sigma_3})\\
\int_T  \b\phi_p^3\cdot\b\phi_q^4\,\mathrm{dx} = &\;
\sum_{\b\tau\in\Sigma_{\ell^p}}\sum_{s=1}^4
 \Mab{\b\alpha^p+\b e_{\tau_1},\b\alpha^q-\b e_s}(\delta_{s\ell^q}(n+2)-\alpha^q_s)
\nonumber\\
&\;\hspace{0.5cm}({\alpha^p_{\sigma_1}+1})
(\alpha^p_{\tau_3}\Szero{\tau_2,s}\!\!-
\alpha^p_{\tau_2}\Szero{\tau_3,s})
 \end{align}
The last block:
 \begin{align}
\int_T  \b\phi_p^4\cdot\b\phi_q^4\,\mathrm{dx} = &\!\!
\sum_{t=1}^4 \!
\sum_{s=1}^4\!
\Mab{\b\alpha^q-\b e_t,\b\alpha^q-\b e_s}
(\delta_{t\ell^p}(n+2)-\alpha^p_t)(\delta_{s\ell^q}(n+2)-\alpha^q_s)\Szero{t,s}
 \end{align}
\end{subequations}

\subsubsection{$\b H(\mathrm{curl})$ stiffness matrix}
Since the curl of Type 2 basis (gradients) vanishes, we only need to calculate $6$ blocks for the stiffness matrix.
The first four blocks:
\begin{subequations}
 \label{curl-stiffness}
 \begin{align}
\int_T \Curl \b\phi_p^1\cdot
\Curl\b\phi_q^1\,\mathrm{dx} = &\;
4\;\Sone{i^pj^p,i^q,j^q}
\\
\int_T  \Curl \b\phi_p^1\cdot \Curl\b\phi_q^2\,\mathrm{dx} = &\;
0\\
\int_T \Curl \b\phi_p^1\cdot\Curl\b\phi_q^3\,\mathrm{dx} = &\;
\frac{12}{(n+2)(n+3)}
\sum_{\b\sigma\in\Sigma_{\ell^q}}
({\alpha^q_{\sigma_1}+1})\nonumber\\
\Big(\alpha^q_{\sigma_3}(\Sone{i^pj^p,\sigma_1\sigma_2}&+
\Sone{i^pj^p,\sigma_3\sigma_2})
-\alpha^q_{\sigma_2}(\Sone{i^pj^p,\sigma_1\sigma_3}+\Sone{i^pj^p,\sigma_2\sigma_3})\Big)
\\
\int_T \Curl \b\phi_p^1\cdot\Curl\b\phi_q^4\,\mathrm{dx} = &\;
0
 \end{align}
The next three blocks are all zeros:
 \begin{align*}
\int_T  \Curl \b\phi_p^2\cdot \Curl\b\phi_q^2\,\mathrm{dx} = &\;
0\\
\int_T  \Curl \b\phi_p^2\cdot \Curl\b\phi_q^3\,\mathrm{dx} = &\;
0\\
\int_T  \Curl \b\phi_p^2\cdot \Curl\b\phi_q^4\,\mathrm{dx} = &\;
0
 \end{align*}
The next two blocks:
 \begin{align}
\int_T \Curl \b\phi_p^3\cdot\Curl\b\phi_q^3\,\mathrm{dx} = &\;
(n+1)^2
\sum_{\b\tau\in\Sigma_{\ell^p}} \sum_{\ell=1}^3
\sum_{\b\sigma\in\Sigma_{\ell^q}} \sum_{k=1}^3
\Mab{\b\alpha^p+\b e_{\tau_1}-\b e_{\tau_\ell},\b\alpha^q+\b e_{\sigma_1}-\b e_{\sigma_k}}\nonumber\\
\;\;\;
({\alpha^p_{\tau_1}+1})&({\alpha^q_{\sigma_1}+1})
(\alpha^p_{\tau_3}\alpha^q_{\sigma_3}\Sone{\tau_\ell\tau_2,\sigma_k\sigma_2}-
\alpha^p_{\tau_3}\alpha^q_{\sigma_2}\Sone{\tau_\ell\tau_2,\sigma_k\sigma_3}\nonumber\\
&\;\;\quad\quad-
\alpha^p_{\tau_2}\alpha^q_{\sigma_3}\Sone{\tau_\ell\tau_3,\sigma_k\sigma_2}
+\alpha^p_{\tau_2}\alpha^q_{\sigma_2}\Sone{\tau_\ell\tau_3,\sigma_k\sigma_3})
\\
\int_T \Curl \b\phi_p^3\cdot\Curl\b\phi_q^4\,\mathrm{dx} = &\;
(n+2)(n+1)^2
\sum_{\b\tau\in\Sigma_{\ell^p}}\sum_{\ell=1}^3
\sum_{i=1}^4 \Mab{\b\alpha^p+\b e_{\tau_1}-\b e_{\tau_\ell},\b\alpha^q-\b e_{\ell^q}-\b e_i}\nonumber\\
&\;({\alpha^p_{\tau_1}+1})
(\alpha^p_{\tau_3}\Sone{\tau_\ell\tau_2,i\ell^q}-
\alpha^p_{\tau_2}\Sone{\tau_\ell\tau_3,i\ell^q})
 \end{align}
The last block:
 \begin{align}
\int_T \Curl \b\phi_p^4\cdot\Curl\b\phi_q^4\,\mathrm{dx} = &\;
\sum_{j=1}^4\sum_{i=1}^4 (n+2)^2(n+1)^2\nonumber\\
&\;
\;\;\;\Mab{\b\alpha^p-\b e_{\ell^p}-\b e_j,\b\alpha^q-\b e_{\ell^q}-\b e_i}
\Sone{j\ell^p,i\ell^q}
 \end{align}
\end{subequations}

\subsection{Mass and stiffness matrices for $\b H(\mathrm{div})$ basis functions}
Now, let us give formula for the mass and stiffness matrices for the $\b H(\mathrm{div})$ basis functions.
For ease of notation, suppose that we have ordered the basis functions such that
\begin{align}
 \label{div-ordering}
 \b\psi_p^1 = \b\chi_{\ell^p},\;\;
 \b\psi_p^2 = \Curl\Xa{F_{\ell^p}T,n}{\b\alpha^p},\;\;
 \b\psi_p^3 = \Curl\Xb{T,n+1}{\ell^p,\b\alpha^p},\;\;
 \b\psi_p^4 = \Xc{n}{\b\alpha},
\end{align}
where $p\in \mathbb{N}^+$ is an integer and the superscript $p$ above indicates the corresponding indices are determined by $p$.

Due to the heterogeneity of the basis functions, the mass and stiffness matrices consists of $4\times 4$ blocks. 
By symmetry, we need to work out the computation for the $10$ upper diagonally  blocks for each matrix. 
Our calculation is given below.

\subsubsection{$\b H(\mathrm{div})$ mass matrix}
Since the Type 2 and Type 3 $\b H(\mathrm{div})$ basis is the curl of Type 3 and Type 4 $\b H(\mathrm{curl})$ basis, the corresponding matrix matrix blocks is identical to the related stiffness matrix for $\b H(\mathrm{curl})$ basis. Hence, we actually only need to compute 
$10-3=7$ blocks this time.

The first four blocks:
\begin{subequations}
 \label{div-mass}
 \begin{align}
\int_T  \b\psi_p^1\cdot\b\psi_q^1\,\mathrm{dx} = &\;
\sum_{\b\tau\in \Sigma_{\ell^p}}
\sum_{\b\sigma\in \Sigma_{\ell^q}}
\Mab{\b e_{\tau_1},\b e_{\sigma_1}}\Sone{\tau_2\tau_3,\sigma_2\sigma_3}\\
\int_T  \b\psi_p^1\cdot \b\psi_q^2\,\mathrm{dx} = &
(n+1)\sum_{\b\tau\in \Sigma_{\ell^p}}
\sum_{\b\sigma\in\Sigma_{\ell^q}} \sum_{k=1}^3
\Mab{\b e_{\tau_1},\b\alpha^q+\b e_{\sigma_1}-\b e_{\sigma_k}}\nonumber\\
&\;\;\;
({\alpha^q_{\sigma_1}+1})
(\alpha^q_{\sigma_3}\Sone{\tau_2\tau_3,\sigma_k\sigma_2}-
\alpha^q_{\sigma_2}\Sone{\tau_2\tau_3,\sigma_k\sigma_3})\\
\int_T \b\psi_p^1\cdot\b\psi_q^3\,\mathrm{dx} = &\;
(n+2)(n+1)\sum_{\b\tau\in \Sigma_{\ell^p}}
\sum_{i=1}^4\Mab{\b e_{\tau_1},\b \alpha^q-\b e_{\ell^q}-\b e_i}
\Sone{\tau_2\tau_3,i\ell^q}\\
\int_T \b\psi_p^1\cdot\b\psi_q^4\,\mathrm{dx} = &\;
\sum_{\b\tau\in \Sigma_{\ell^p}}
\sum_{i=1}^4
\sum_{\b\sigma\in \Sigma_{i}}
(-1)^{i+1}(\alpha_i^q+1)\alpha_{\sigma_1}^q\Mab{\b e_{\tau_1},\b \alpha^q+\b e_i}
\Sone{\tau_2\tau_3,\sigma_2\sigma_3}
 \end{align}
The next three blocks:
 \begin{align}
 \int_T \b\psi_p^2\cdot\b\psi_q^2\,\mathrm{dx} = &\;
(n+1)^2
\sum_{\b\tau\in\Sigma_{\ell^p}} \sum_{\ell=1}^3
\sum_{\b\sigma\in\Sigma_{\ell^q}} \sum_{k=1}^3
\Mab{\b\alpha^p+\b e_{\tau_1}-\b e_{\tau_\ell},\b\alpha^q+\b e_{\sigma_1}-\b e_{\sigma_k}}\nonumber\\
\;\;\;
({\alpha^p_{\tau_1}+1})&({\alpha^q_{\sigma_1}+1})
(\alpha^p_{\tau_3}\alpha^q_{\sigma_3}\Sone{\tau_\ell\tau_2,\sigma_k\sigma_2}-
\alpha^p_{\tau_3}\alpha^q_{\sigma_2}\Sone{\tau_\ell\tau_2,\sigma_k\sigma_3}\nonumber\\
&\;\;\quad\quad-
\alpha^p_{\tau_2}\alpha^q_{\sigma_3}\Sone{\tau_\ell\tau_3,\sigma_k\sigma_2}
+\alpha^p_{\tau_2}\alpha^q_{\sigma_2}\Sone{\tau_\ell\tau_3,\sigma_k\sigma_3})
\\
\int_T \b\psi_p^2\cdot \b\psi_q^3\,\mathrm{dx} = &\;
(n+2)(n+1)^2
\sum_{\b\tau\in\Sigma_{\ell^p}}\sum_{\ell=1}^3
\sum_{i=1}^4 \Mab{\b\alpha^p+\b e_{\tau_1}-\b e_{\tau_\ell},\b\alpha^q-\b e_{\ell^q}-\b e_i}\nonumber\\
&\;({\alpha^p_{\tau_1}+1})
(\alpha^p_{\tau_3}\Sone{\tau_\ell\tau_2,i\ell^q}-
\alpha^p_{\tau_2}\Sone{\tau_\ell\tau_3,i\ell^q})\\
\int_T \b\psi_p^2\cdot\b\psi_q^4\,\mathrm{dx} = &\;
(n+1)
\sum_{\b\tau\in\Sigma_{\ell^p}}\sum_{\ell=1}^3
\sum_{i=1}^4
\sum_{\b\sigma\in \Sigma_{i}}
(-1)^{i+1}\Mab{\b\alpha^p+\b e_{\tau_1}-\b e_{\tau_\ell},\b\alpha^q+\b e_i}\nonumber\\
&\;({\alpha^p_{\tau_1}+1})(\alpha_i^q+1)\alpha_{\sigma_1}^q
(\alpha^p_{\tau_3}\Sone{\tau_\ell\tau_2,\sigma_2\sigma_3}-
\alpha^p_{\tau_2}\Sone{\tau_\ell\tau_3,\sigma_2\sigma_3})
 \end{align}
The next two blocks:
 \begin{align}
 \int_T \b\psi_p^3\cdot\b\psi_q^3\,\mathrm{dx} = &\;
\sum_{j=1}^4\sum_{i=1}^4 (n+2)^2(n+1)^2\nonumber\\
&\;
\;\;\;\Mab{\b\alpha^p-\b e_{\ell^p}-\b e_j,\b\alpha^q-\b e_{\ell^q}-\b e_i}
\Sone{j\ell^p,i\ell^q}\\
\int_T \b\psi_p^3\cdot\b\psi_q^4\,\mathrm{dx} = &\;
(n+2)(n+1)\sum_{j=1}^4 \sum_{i=1}^4
\sum_{\b\sigma\in \Sigma_{i}}
(-1)^{i+1}
\Mab{\b\alpha^p-\b e_{\ell^p}-\b e_j,\b\alpha^q+\b e_i}\nonumber\\
&\;(\alpha_i^q+1)\alpha_{\sigma_1}^q
\Sone{j\ell^p,\sigma_2\sigma_3}
 \end{align}
The last block:
 \begin{align}
\int_T \b\psi_p^4\cdot\b\psi_q^4\,\mathrm{dx} = &\;
\sum_{j=1}^4
\sum_{\b\tau\in \Sigma_{j}}\sum_{i=1}^4
\sum_{\b\sigma\in \Sigma_{i}}
(-1)^{j+i}
\Mab{\b\alpha^p+\b e_j,\b\alpha^q+\b e_i}\nonumber\\
&\;(\alpha_j^p+1)\alpha_{\tau_1}^p(\alpha_i^q+1)\alpha_{\sigma_1}^q
\Sone{\tau_2\tau_3,\sigma_2\sigma_3}
 \end{align}
\end{subequations}

\subsubsection{$\b H(\mathrm{div})$ stiffness matrix}
This time, we only need to compute three blocks since the divergence of Type 2 and Type 3 basis functions are zero.
\begin{subequations}
 \label{div-stiffness}
 \begin{align}
\int_T \Div \b\psi_p^1\,
\Div\b\psi_q^1\,\mathrm{dx} = &\;
(-1)^{\ell^p+\ell^q}9\;\Stwo{123}\\
\int_T \Div \b\psi_p^1\,\Div\b\psi_q^4\,\mathrm{dx} = &\;0\\
\int_T \Div \b\psi_p^4\,\Div\b\psi_q^4\,\mathrm{dx} = &\;
(n+1)^2\sum_{j=1}^4\sum_{\ell=1}^4 \sum_{i=1}^4\sum_{k=1}^4 (-1)^{i+j}\Mab{\b \alpha^p+\b e_j-\b e_\ell,\b \alpha^q+\b e_i-\b e_k}\nonumber\\
&\;(\alpha_j^p+1)(\delta_{j\ell}n-\alpha_\ell)(\alpha_i^q+1)(\delta_{ik}n-\alpha_k)
\Stwo{123}
 \end{align}
\end{subequations}

\section{Numerical Examples}
\label{sec:num}
We conclude with three examples which illustrate the use of the bases. 
We consider the case of 
constant data and affine tetrahedral elements although the ideas presented in \cite{AinsworthAndriamaroDavydov11} could be used to treat much more general scenarios. 
The stiffness and mass matrices 
are explicitly computed using the formula in 
Section \ref{sec:comp} (along with proper ordering of degrees of freedom); while the 
integrals involving
source terms are obtained via Stroud conical quadrature rules \cite{AinsworthAndriamaroDavydov11}. 

The first two examples were taken from \cite{Kirby14} where, for the first time,  $\b H(\mathrm{curl})$ and $\b H(\mathrm{div})$ Bernstein
basis in \cite{ArnoldFalkWinther09} were implemented.
The third example solve a mixed elliptic problem with a divergence-free vector solution, where we 
exploit the fact that our basis provides  a clear separation of the divergence free functions to
solve the problem with a 
reduced $\b H{(\mathrm{div})}$ basis whereby all the Type 4 (not divergence-free) $\b H(\mathrm{div})$ bubbles are omitted, coupled 
with a piecewise constant scalar filed. 
The reduced basis approach produce exactly the same vector field approximation but requires fewer degrees
 of freedom
(see e.g. \cite{Lehrenfeld:10,LehrenfeldSchoberl16} on the use of reduced basis for incompressible flow problems).

\subsection{Maxwell eigenvalue problem}
We consider the problem of finding resonances $\omega\in \mathbb{R}$ 
and eigenfunctions $\b E\in\b H_0(\mathrm{curl})$ 
such  that 
\begin{align*}
 (\Curl \b E,\Curl \b F)_\Omega = \omega^2(\b E,\b F)_\Omega,\quad\quad
 \forall \b F\in\b H_0(\mathrm{curl}).
\end{align*}
Same as in \cite{Kirby14}, we consider the cube $\Omega =[0, \pi]^3$ divided into 6 tetrahedra and study $p$-convergence of the 
finite element eigenvalue problem with \NED's $\b H(\mathrm{curl})$ finite elements of first kind, $\NDOne{n}$, 
using basis functions in Table  \ref{table-geo}.
For this problem, the first eigenvalues are $2$ with multiplicity $3$, 3 with multiplicity $2$, and $5$ with multiplicity $6$.
We take the first $11$ nonzero computed eigenvalues 
and measure the average error in each one in Figure \ref{fig:cavity}. Similar convergence behavior as that in \cite[Fig. 4.8]{Kirby14} is obtained; in particular, we see that ten-digit accuracy in all three of the first eigenvalues is obtained when we reach degree 12, we also see round-off effect when degree is 13. 
\begin{figure}[!ht]
  \scalebox{0.30}[0.26]{\includegraphics{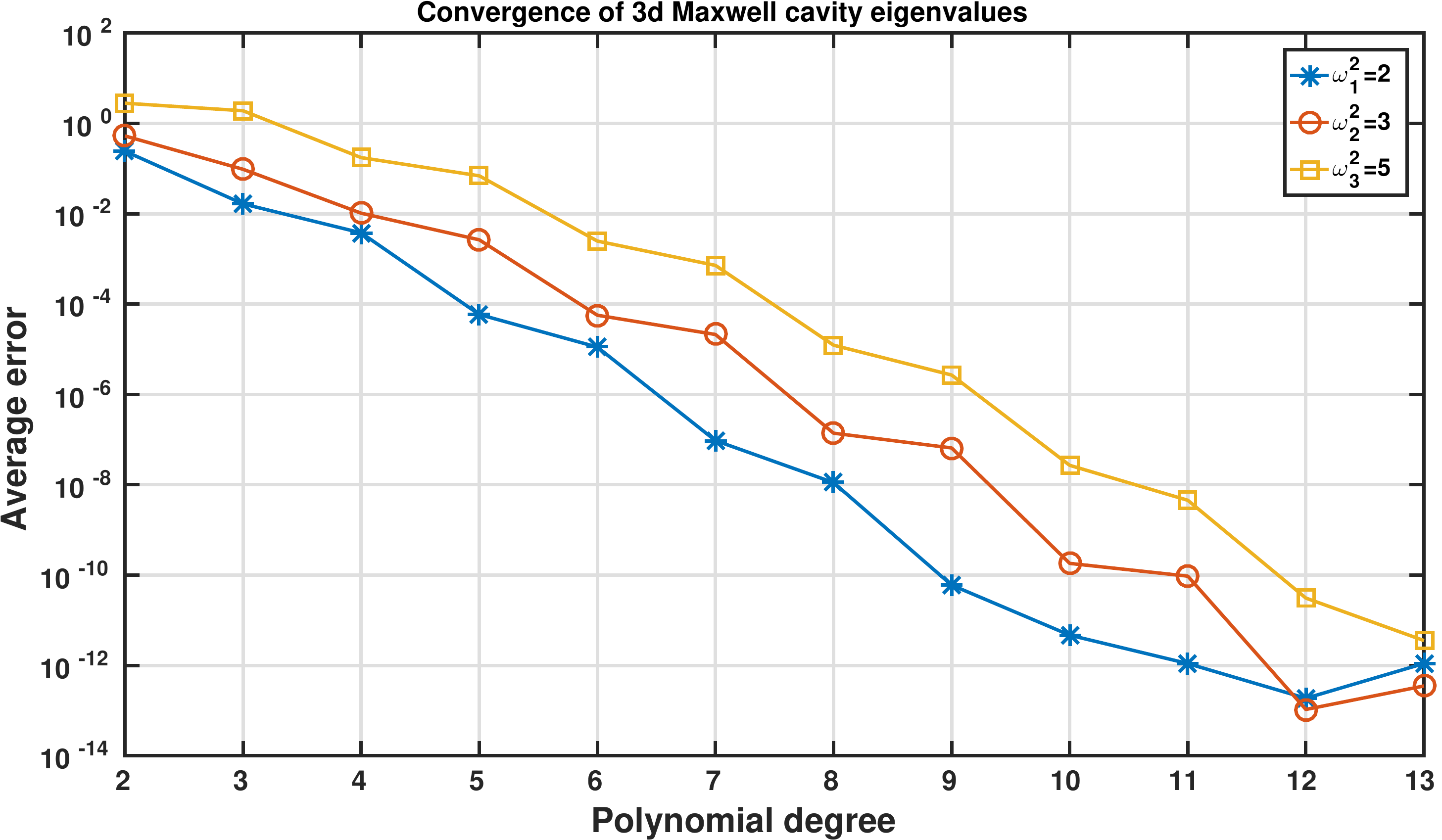}}
  \caption{Error in the first three non-zero distinct eigenvalues for the three-dimensional cavity resonator problem}\label{fig:cavity}
\end{figure}

\subsection{Mixed Poisson problem}
Here we consider the mixed Poisson problem
\begin{align*}
 \b u +\Grad p = &\; 0\\
 \Div\b u = &\; f
\end{align*}
on a unit cube which homogeneous Dirichlet boundary conditions.
We pick $f$ such that the true solution is $p(x,y,z) =\sin(\pi x)\sin(\pi y)\sin(\pi z)$.
Again, we divided the unit cube into 6 tetrahedra and study $p$-convergence of the mixed finite elements:
find $(\b u_h,p_h)\in \VVV_h\times \WWW_h$ such that 
\begin{align*}
 (\b u_h,\b v_h)_\Omega -(p_h,\Div \b v_h)_\Omega = &\; 0\\
( \Div\b u,q_h)_\Omega = &\; (f,q_h)_\Omega
\end{align*}
the stable finite element pair $\RT{n}\times \pol_n$ using the basis functions in Table \ref{table-geo} 
is used in the computation. 
We use static condensation to condense out the interior degrees of freedom for the velocity and 
{\it average-zero} degrees of freedom for the potential, the resulting condensed system matrix consists of  
face degrees of freedom for velocity and piecewise constants for pressure.
The size of the condensed system matrix, along with that of the original one, for various polynomial degree is listed in Table \ref{table:size-mp}. 
We observe significant computational saving in terms of the global system matrix size by the static condensation.
We remark that such static condensation relies on  the fact that there are no inter-element communication for the 
internal degrees of freedom for velocity and average-zero degrees of freedom for the potential, hence they can be static condensed out using the face degrees of freedom for velocity and 
constant degree of freedom for potential in each element. 
If the potential degrees of freedom were not to have the 
constant and average-zero separation, as the basis used in \cite{Kirby14}, one might not be able to condense out any of the pressure degrees of freedom.
\begin{table}[!ht]
\caption{Size of condensed system matrix $K_c$ and original system matrix $K_o$ for mixed Poisson problem
on a six-tetrahedron mesh for polynomial degree 0 to 14}
\begin{tabular}{|l| c| c| c| c| c| c| c| c |}
\hline
degree & 0 & 1 & 2 &3 & 4 & 5 & 6 & 7\\
\hline
size $K_c$ & 24 & 60 & 114 &186 & 276 & 384 & 510 & 654\\
size $K_o$ & 24 & 96 & 240 &480 & 840 & 1344&2016&2880\\
\hline
degree & 8 & 9 & 10 &11 & 12 & 13 & 14 &\\
\hline
size $K_c$& 816 & 996&1194&1410&1644&1896&2166&\\ 
size $K_o$& 3960&5280&6864&8736&10920&13440&16320&\\
\hline
\end{tabular}
\label{table:size-mp}
\end{table}

The $L^2$-error for the potential and velocity fields are plotted in 
Figure \ref{fig:darcy}. 
We observe exponential convergence as expected. 
We also see similar jagged convergence in potential, with small drops from even up to odd degree followed by large drops 
from odd up to even degree,
as documented in \cite{Kirby14}. However, such jagged convergence behavior seems do not appear for the velocity approximation.
\begin{figure}[!ht]
  \scalebox{0.30}[0.26]{\includegraphics{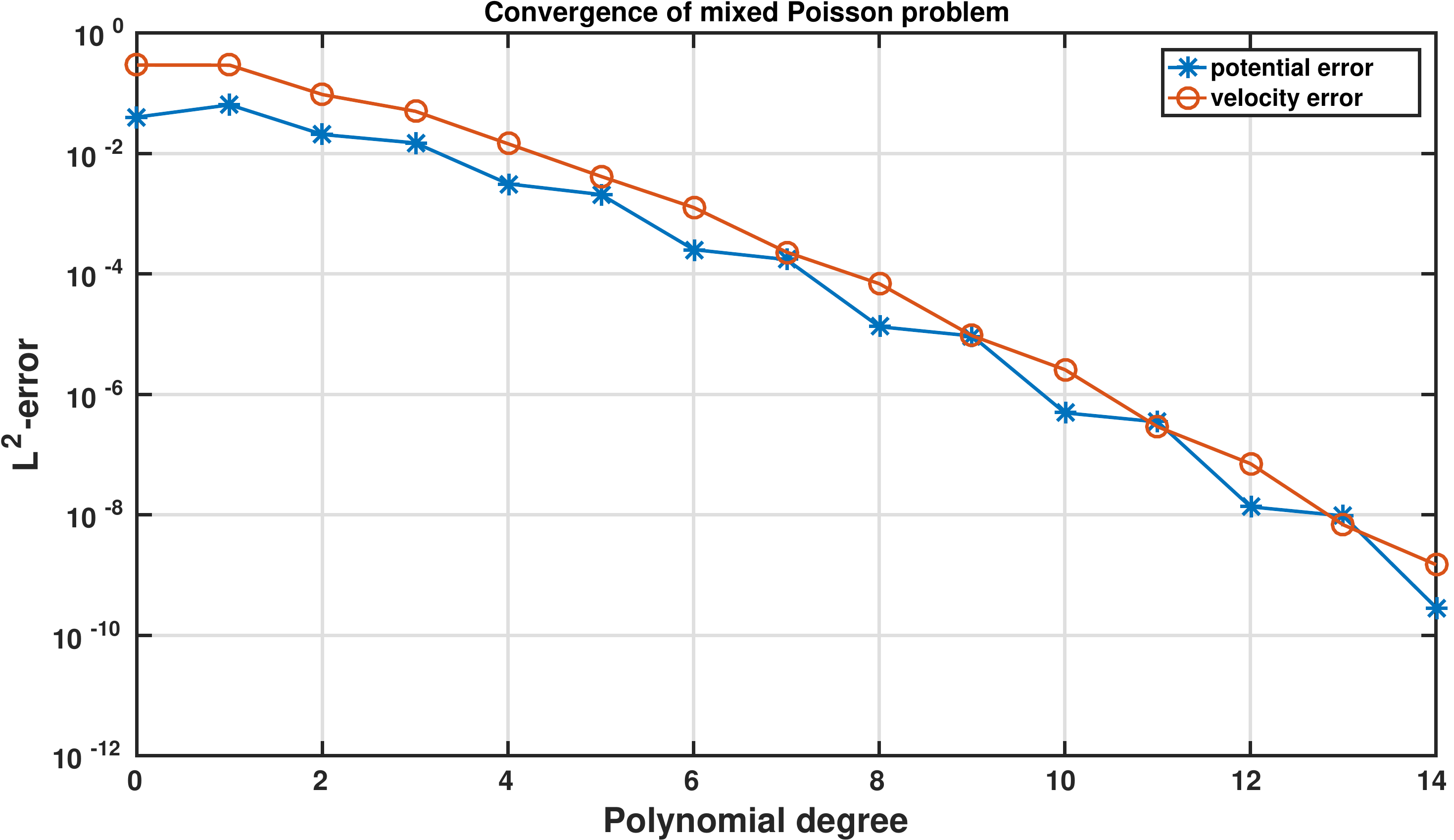}}
  \caption{Convergence of the mixed finite element method as a function of the polynomial degree on a mesh with six tetrahedra.}\label{fig:darcy}
\end{figure}


\subsection{Modified mixed Poisson problem}
Here we consider the following modified mixed Poisson problem
\begin{align*}
 \b u +\Grad p = &\; \b f\\
 \Div\b u = &\; 0
\end{align*}
on a unit cube with homogeneous Dirichlet boundary conditions for scalar field $p$.
We pick $\b f$ such that the true solution is $p(x,y,z) =\sin(\pi x)\sin(\pi y)\sin(\pi z)$ and 
$\b u(x,y,z) = (0, \sin(\pi x)\cos(\pi y)\sin(\pi z),-\sin(\pi x)\sin(\pi y)\cos(\pi z))^T$.

This time, we study $h$-convergence of the mixed finite elements:
find $(\b u_h,p_h)\in \VVV_h\times \WWW_h$ such that 
\begin{align*}
 (\b u_h,\b v_h)_\Omega -(p_h,\Div \b v_h)_\Omega = &\; (\b f, \b v_h)_\Omega\\
( \Div\b u,q_h)_\Omega = &\; 0.
\end{align*}
In Figure \ref{fig:mdarcy}, we plot  $L^2$-error for the vector field for two finite element methods 
for polynomial degree $1,3,5$ on five consecutive meshes 
obtained by cutting the unit cube to $m^3$ congruent cubes and dividing each cube into 6 tetrahedra, we vary 
the (inverse) mesh size $m$ from 2 to 10. 
The first method, labeled as {\sf Div-Free:NO} in Figure \ref{fig:mdarcy}, uses the Raviart-Thomas finite elements used in the previous test:
\[
 \RT{n}\times \pol_n = 
 \left( \VVV_{low}\oplus_{F\in\mathcal{F}(K)} \Curl\EEE_n^F\oplus \Curl\EEE_{n+1}^T
 \oplus \VVV_n^T\right)\times \left(\WWW_{low}\oplus \Div\VVV_n^T\right),
\]
the second method, labeled as {\sf Div-Free:YES} in Figure \ref{fig:mdarcy}, uses the  following reduced (divergence-free) finite elements:
\[
\left( \VVV_{low}\oplus_{F\in\mathcal{F}(K)} \Curl\EEE_n^F\oplus \Curl\EEE_{n+1}^T\right)\times \WWW_{low}.
\]
See definition of the above spaces in Section \ref{sec:sequence}.
It is easy to show that both methods produce the same vector approximation, which is clearly observed in Figure \ref{fig:mdarcy},
where the optimal $L^2$-convergence error for the vector field for both methods are on top of each other.
We list the global (static-condensed) number of degrees of freedom along with 
the local number of degrees of freedom
 for both methods in Table \ref{table:size-mpx}. 
 While both methods share the same number of global degrees of freedom, 
the reduced basis approach has significant amount of local degrees of freedom reduction.
\begin{figure}[!ht]
  \scalebox{0.35}[0.3]{\includegraphics{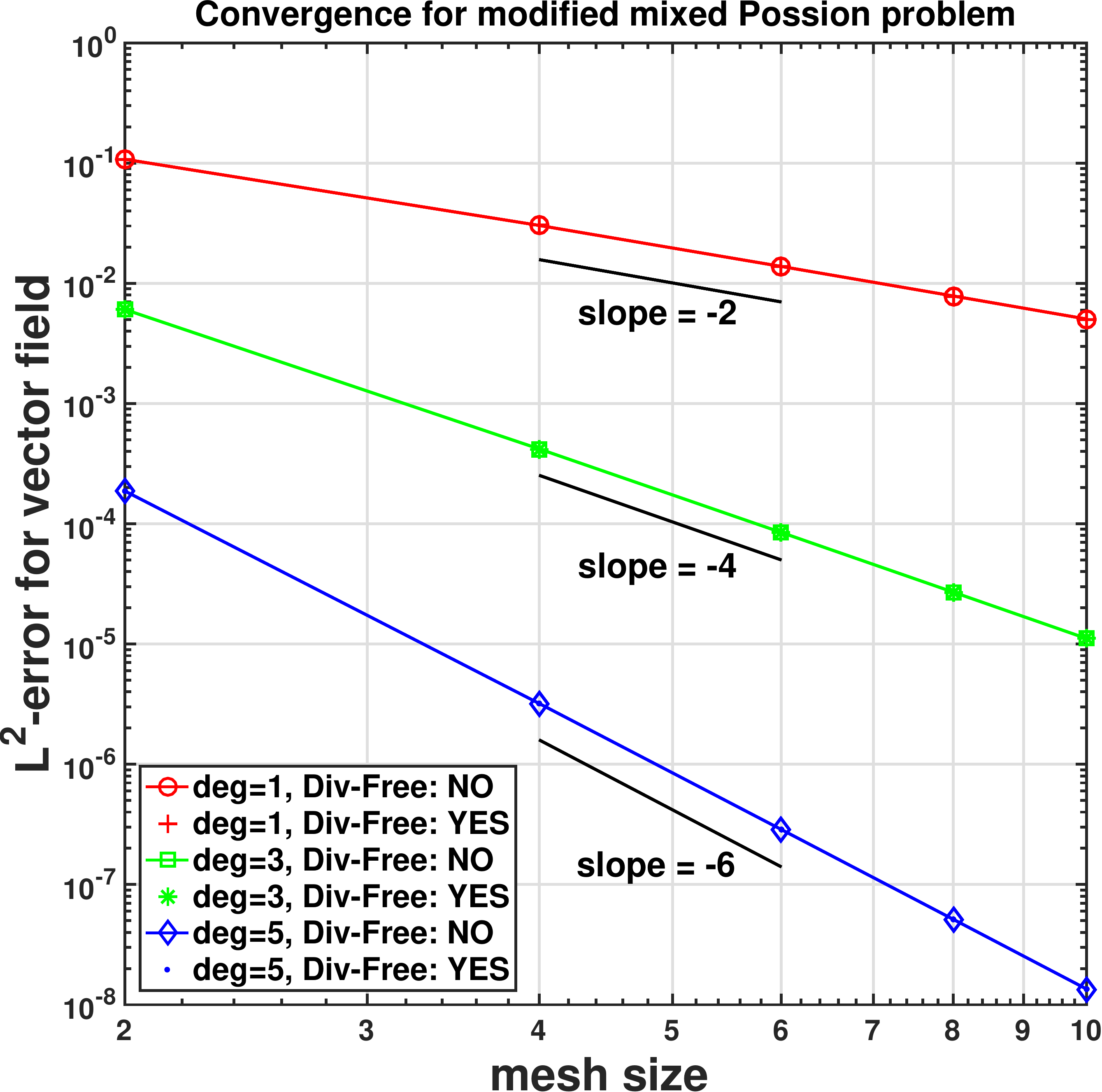}}
  \caption{$L^2$-error for vector field of the modified mixed Poisson problem. {\sf Div-Free:NO} is referred as 
  the first (original) method, and {\sf Div-Free:YES} is referred as 
  the second (reduced) method.}
  \label{fig:mdarcy}
\end{figure}

\begin{table}[!ht]
\caption{Number of degrees of freedom for the two methods. $N_g$: global number of dofs for both methods, 
$N_l^{red}$: local number of dofs for second method (with reduced basis), $N_l$: local number of dofs for first method.}
\resizebox{0.9\textwidth}{!}
{
\begin{tabular}{|l | c c c| c c c| c c  c|}
\hline
Nelems & &deg=1 &&& deg=2 && &deg=3&\\
\hline
 &$N_g$ &$N_l^{red}$ &$N_l$&$N_g$ &$N_l^{red}$ &$N_l$&$N_g$ &$N_l^{red}$ &$N_l$\\
\hline
48     & 408     & 0 & 288   & 1248 & 528 & 2352 &2568 & 2400 & 7680\\
384   & 2976   & 0 & 2304 &9024 & 4224 & 18816& 18528 & 19200 & 61440\\
1296 & 9720   & 0 & 7776 &29376 & 14256 & 63504 & 60264 & 64800&207360\\
3072 & 22656 & 0&18432&68352&33792&150528&140160&153600&491520\\ 
6000 & 43800 & 0&36000&132000&66000&294000&270600&300000&960000\\
\hline
\end{tabular}
}
\label{table:size-mpx}
\end{table}


\section*{Appendix: Proof of Lemma \ref{lemma:cell-curl-bubble}}
\begin{proof}
Firstly, we show that $\Xb{n+1}{\ell,\b{\alpha}}$ has vanishing tangential
components on $\partial T$ for $\b{\alpha}\in\mathring{\c{I}}_{n+2}(T)$. The
vector field $\BB{n+1}{\b{\alpha}-\b{e}_\ell}\Grad\lambda_\ell$ has vanishing
tangential components on the face $F_\ell$ since $\Grad\lambda_\ell$
points in the direction of the normal on face $F_\ell$, whilst the Bernstein
polynomial $\BB{n+1}{\b{\alpha}-\b{e}_\ell}$ vanishes on the remaining faces
since $\b{\alpha}\in\mathring{\c{I}}_{n+2}(T)$. Likewise, the Bernstein
polynomial $\BB{n+2}{\b{\alpha}}$ belongs to $H^1_0(T)$ and so 
$\Grad\BB{n+2}{\b{\alpha}}\in\b{H}_0(\mathrm{curl}; T)$. Hence, 
$\Xb{n+1}{\ell,\b{\alpha}}\in\b{H}_0(\mathrm{curl}; T)$. 

Concerning membership of the \NED space $\NDOne{n}$, we begin with the
observation that $\Xb{n+1}{\ell,\b{\alpha}}\in\mathbb{P}_{n+1}^3$.
Consequently, in view of~\cite[Proposition 1]{Nedelec80}, it suffices to
show that $\b{x}\cdot\Xb{n+1}{\ell,\b{\alpha}}\in\mathbb{P}_{n+1}$, where
$\b{x}\in\RR^3$ denotes the usual position vector of points in $\RR^3$. Next,
since the barycentric coordinates are affine functions of $\b{x}$, 
\begin{equation*}
  (\b{x}-\b{x}_\ell)\cdot\Grad\lambda_m
  = \lambda_m(\b{x})-\lambda_m(\b{x}_\ell)
  = \lambda_m(\b{x})-\delta_{\ell m}.
\end{equation*}
Hence, 
\begin{equation*}
  (\b{x}-\b{x}_\ell)\cdot\BB{n+1}{\b{\alpha}-\b{e}_\ell}
  \Grad\lambda_\ell
  = (\lambda_\ell - 1)\BB{n+1}{\b{\alpha}-\b{e}_\ell}
  = \frac{\alpha_\ell}{n+2}\BB{n+2}{\b{\alpha}}
  - \BB{n+1}{\b{\alpha}-\b{e} _\ell},
\end{equation*}
thanks to the following elementary property of Bernstein polynomials 
\begin{equation*}
  \lambda_m\BB{n-1}{\b{\alpha}-\b{e}_m}
  = \dfrac{\alpha_m}{n+2}\BB{n+2}{\b{\alpha}}.
\end{equation*}
The same property, along with identity
$
    \Grad B^{n+1}_\b{\alpha} = (n+1)\sum_{k=1}^4
B^n_{\b{\alpha}-\b{e}_k}\Grad\lambda_k,
$
 also gives 
\begin{align*}
  \dfrac{1}{n+2}(\b{x}-\b{x}_\ell)\cdot\Grad\BB{n+2}{\b{\alpha}}
  &= \sum_{m=1}^4\BB{n+1}{\b{\alpha}-\b{e}_m}
    (\b{x}-\b{x}_\ell)\cdot\Grad\lambda_m \\
  &= \sum_{m=1}^4\lambda_m\BB{n+1}{\b{\alpha}-\b{e}_m}
   - \BB{n+1}{\b{\alpha}-\b{e}_\ell}\\
  &= \BB{n+2}{\b{\alpha}}-\BB{n+1}{\b{\alpha}-\b{e}_\ell},
\end{align*}
on recalling $\sum_{m=1}^4\alpha_m = n+2$. Therefore, 
\begin{equation*}
  (\b{x}-\b{x}_\ell)\cdot\Xb{n+1}{\ell,\b{\alpha}}
  = -(n+2-\alpha_\ell)\BB{n+1}{\b{\alpha}-\b{e}_\ell},
\end{equation*}
or, equally well,
$\b{x}\cdot\Xb{n+1}{\ell,\b{\alpha}}\in\mathbb{P}_{n+1}$. 

A simple calculation implies that 
the total number of functions in the sets in~\eqref{BBC-bubble} is
$2{n+1\choose3}+{n\choose2}$, which is equal to the dimension of divergence-free
bubble space on a tetrahedral 
$\b H_0(\mathrm{div}^0;T)\cap \mathbb{P}_n^3$.
By exact sequence property, the fact that 
$\Xb{n+1}{\ell,\b{\alpha}}\in \NDOne{n}$ implies that 
$\Curl\Xb{n+1}{\ell,\b{\alpha}}\in 
\b H_0(\mathrm{div}^0;T)\cap \mathbb{P}_n^3.
$

Let us now prove the linear independence 
of the functions in~\eqref{BBC-bubble-curl}.

Clearly 
\begin{equation*}
  \Curl\Xb{n+1}{\ell,\b{\alpha}}
  = (n+2)\Curl(\BB{n+1}{\b{\alpha}-\b{e}_\ell}\Grad\lambda_\ell)
\end{equation*}
then, using $
    \Grad B^{n+1}_\b{\alpha} = (n+1)\sum_{k=1}^4
B^n_{\b{\alpha}-\b{e}_k}\Grad\lambda_k,
$, we obtain
\begin{equation*}
  \frac{1}{(n+1)(n+2)}\Curl\Xb{n+1}{\ell,\b{\alpha}}
  = \sum_{m=1}^4 \BB{n}{\b{\alpha}-\b{e}_\ell-\b{e}_m}
  \Grad\lambda_m\times\Grad\lambda_\ell.
\end{equation*}
In particular if $\ell=1$, then 
\begin{align*}
\dfrac{1}{(n+1)(n+2)}\Curl\Xb{n+1}{1,\b{\alpha}} 
    &=\BB{n}{\b{\alpha}-\b{e}_1-\b{e}_2}\Grad\lambda_2\times\Grad\lambda_1 \\
    &+\BB{n}{\b{\alpha}-\b{e}_1-\b{e}_3}\Grad\lambda_3\times\Grad\lambda_1 \\
    &+\BB{n}{\b{\alpha}-\b{e}_1-\b{e}_4}\Grad\lambda_4\times\Grad\lambda_1
\end{align*}
and eliminating $\Grad\lambda_4$ with the aid of the expression
\begin{equation*}
  \Grad\lambda_1 + \Grad\lambda_2 + \Grad\lambda_3 + \Grad\lambda_4 = \b{0}
\end{equation*}
gives 
\begin{align}\label{eq:curlQ}
\dfrac{1}{(n+1)(n+2)}\Curl\Xb{n+1}{1,\b{\alpha}}
    &=( \BB{n}{\b{\alpha}-\b{e}_1-\b{e}_4} 
     - \BB{n}{\b{\alpha}-\b{e}_1-\b{e}_2} )
     \Grad\lambda_1\times\Grad\lambda_2 \nonumber\\
    &+( \BB{n}{\b{\alpha}-\b{e}_1-\b{e}_3} 
     - \BB{n}{\b{\alpha}-\b{e}_1-\b{e}_4} )   
     \Grad\lambda_3\times\Grad\lambda_1.
\end{align}
Analogous expressions are readily obtained for
$\Curl\Xb{n+1}{\ell,\b{\alpha}}$, $\ell=2,3$ by cyclic permutation of the
indices 1-2-3. 

To simplify notation, we denote the direction vectors 
\begin{equation*}
\b t_{12}:=  \Grad\lambda_1\times\Grad\lambda_2,\;\;
\b t_{23}:=  \Grad\lambda_2\times\Grad\lambda_3,\;\;
\b t_{31}:=  \Grad\lambda_3\times\Grad\lambda_1.
\end{equation*}
It is clear that $\b t_{12}, \b t_{23}, \b t_{31}\in \RR^3$
are linearly independent. 

Now, let coefficients $\ca$, 
$\cb$, and $\cc$ be such that 
\[
 \sum_{\b\alpha\in \mathring{\c I}_{3}^{n+2}}
 \ca\Curl\Xb{n+1}{1,\b{\alpha}}
 +
  \sum_{\b\alpha\in \mathring{\c I}_{3}^{n+2}}
 \cb\Curl\Xb{n+1}{2,\b{\alpha}}
+ \sum_{\overset{\b\alpha\in \mathring{\c I}_{3}^{n+2}}{\alpha_3=1\;\;}}
 \cc\Curl\Xb{n+1}{3,\b{\alpha}}
=\b 0.
\]
By~\eqref{eq:curlQ}, we have
\begin{align*}
 & \left(\sum_{\b\alpha\in \mathring{\c I}_{3}^{n+2}}\!\!\!
 \ca
 ( \BB{n}{\b{\alpha}-\b{e}_1-\b{e}_4}\!\!\!
     - \!\BB{n}{\b{\alpha}-\b{e}_1-\b{e}_2} )
\!+ \!\!\!\!  \sum_{\b\alpha\in \mathring{\c I}_{3}^{n+2}}
\!\!\! \cb
 ( \BB{n}{\b{\alpha}-\b{e}_1-\b{e}_2}\!\!\! 
     -\! \BB{n}{\b{\alpha}-\b{e}_2-\b{e}_4} )
     \!\right)\!
\b t_{12}&&\\
 &+ \left(\sum_{\b\alpha\in \mathring{\c I}_{3}^{n+2}}\!\!\!
 \cb
 ( \BB{n}{\b{\alpha}-\b{e}_2-\b{e}_4}\!\!\!
     - \!\BB{n}{\b{\alpha}-\b{e}_2-\b{e}_3} )
\!+ \!\!\!\!  \sum_{
\overset{\b\alpha\in \mathring{\c I}_{3}^{n+2}}{\alpha_3 =1\;\;}
}
\!\!\! \cc
 ( \BB{n}{\b{\alpha}-\b{e}_2-\b{e}_3}\!\!\! 
     -\! \BB{n}{\b{\alpha}-\b{e}_3-\b{e}_4} )
     \!\right)\!
\b t_{23}&&\\
 &+ \left(\sum_{\b\alpha\in \mathring{\c I}_{3}^{n+2}}\!\!\!
 \ca
 ( \BB{n}{\b{\alpha}-\b{e}_1-\b{e}_3}\!\!\!
     - \!\BB{n}{\b{\alpha}-\b{e}_1-\b{e}_4} )
\!+ \!\!\!\!  \sum_{
\overset{\b\alpha\in \mathring{\c I}_{3}^{n+2}}{\alpha_3 =1\;\;}
}
\!\!\! \cc
 ( \BB{n}{\b{\alpha}-\b{e}_3-\b{e}_4}\!\!\! 
     -\! \BB{n}{\b{\alpha}-\b{e}_1-\b{e}_3} )
     \!\right)\!
\b t_{31}=\b 0.
\end{align*}
This implies that each term behind the direction vectors 
$\b t_{12}$, $\b t_{23}$, and $\b t_{31}$
of the above expression is {\it zero}.

Let us now focus on the term behind $\b t_{31}$.
We have
\begin{equation}
\label{eq-23}
 \sum_{\b\alpha\in \mathring{\c I}_{3}^{n+2}}\!\!\!
 \ca
 ( \BB{n}{\b{\alpha}-\b{e}_1-\b{e}_3}\!\!\!
     - \!\BB{n}{\b{\alpha}-\b{e}_1-\b{e}_4} )
\!+ \!\!\!\!  \sum_{
\overset{\b\alpha\in \mathring{\c I}_{3}^{n+2}}{\alpha_3 =1\;\;}
}
\!\!\! \cc
 ( \BB{n}{\b{\alpha}-\b{e}_3-\b{e}_4}\!\!\! 
     -\! \BB{n}{\b{\alpha}-\b{e}_1-\b{e}_3} )=0
\end{equation}
The goal is to show that all the coefficients in the above expression is {\it zero}.
We proceed the proof by mathematical induction on the last index $\alpha_4$.

Given $k\in \NN_+$, 
if
\begin{subequations}
\label{cc-ca}
 \begin{align}
\label{cc1}
 \cc = &0 \;\text{ for }\; \b\alpha\in \mathring{\c I}_{3}^{n+2}\;\text{ such that } {\alpha_3 =1,\;\alpha_4\le k},\\
\label{ca1}
 \ca = &0 \;\text{ for }\; \b\alpha\in \mathring{\c I}_{3}^{n+2}\;\text{ such that } {\;\alpha_4\le k},
\end{align}
\end{subequations}
let us prove 
\begin{subequations}
\label{cc-ca-1}
\begin{align}
\label{cc1-1}
 \cc = &0 \;\text{ for }\; \b\alpha\in \mathring{\c I}_{3}^{n+2}\;\text{ such that } {\alpha_3 =1,\;\alpha_4= k+1},\\
 \label{ca1-1}
 \ca = &0 \;\text{ for }\; \b\alpha\in \mathring{\c I}_{3}^{n+2}\;\text{ such that } {\;\alpha_4= k+1}.
\end{align}
\end{subequations}

First, by induction hypothesis \eqref{cc-ca}, identity \eqref{eq-23} reduces
to 
\begin{equation}
\label{cc-red}
 \sum_{
 \overset{\b\alpha\in \mathring{\c I}_{3}^{n+2}}{\alpha_4 \ge k+1\;\;}}\!\!\!
 \ca
 ( \BB{n}{\b{\alpha}-\b{e}_1-\b{e}_3}\!\!\!
     - \!\BB{n}{\b{\alpha}-\b{e}_1-\b{e}_4} )
\!+ \!\!\!\!  \sum_{
\overset{\b\alpha\in \mathring{\c I}_{3}^{n+2}}{\alpha_3 =1,
\alpha_4 \ge k+1\;\;}
}
\!\!\! \cc
 ( \BB{n}{\b{\alpha}-\b{e}_3-\b{e}_4}\!\!\! 
     -\! \BB{n}{\b{\alpha}-\b{e}_1-\b{e}_3} )=0
\end{equation}
Recursively using 
$
    \lambda_k B^{n}_{\b{\alpha}} =\frac{\alpha_k+1}{n+1} 
B^{n+1}_{\b{\alpha}+\b{e}_k},
$ we have
\[
\BB{n}{\b{\beta}}=\frac{{n\choose k}\lambda_4^k}{{\beta_4\choose k}}
\BB{n-k}{\b{\beta}-k\b e_4}\;\;
\text{ for }\b\beta\in \INThree{n} \text{ such that }
\beta_4 \ge k.
\]
Applying the above identity to \eqref{cc-red}, and dividing the resulting 
expression by $ {n\choose k}\lambda_4^k$, we obtain
\begin{align}
\label{cc-red2}
0=&\; \sum_{
 \overset{\b\alpha\in \mathring{\c I}_{3}^{n+2}}{\alpha_4 \ge k+1\;\;}}
\left( \frac{\ca}{{\alpha_4\choose k}} \BB{n-k}{\b{\alpha}-\b{e}_1-\b{e}_3-k\b e_4}
     -\frac{\ca}{{\alpha_4-1\choose k}} \BB{n-k}{\b{\alpha}-\b{e}_1-(k+1)\b{e}_4} \right)\nonumber \\
&\;\;\;\;\;\; +\sum_{
\overset{\b\alpha\in \mathring{\c I}_{3}^{n+2}}{\alpha_3 =1,
\alpha_4 \ge k+1\;\;}
}
 \left(\frac{\cc}{{\alpha_4-1\choose k}}  \BB{n-k}{\b{\alpha}-\b{e}_3-(k+1)\b{e}_4}
     -\frac{\cc}{{\alpha_4\choose k}}  \BB{n-k}{\b{\alpha}-\b{e}_1-\b{e}_3-k\b e_4} 
     \right)
\end{align}
Now, evaluating right hand side of \eqref{cc-red2} 
on the edge $E_{34}:=F_3\cap F_4$, and recalling 
that $\lambda_3=\lambda_4=0$ on this edge, we get  
\[
\sum_{
\overset{\b\alpha\in \mathring{\c I}_{3}^{n+2}}{\alpha_3 =1,
\alpha_4 = k+1\;\;}
}{\ca}\BB{n}{\b{\alpha}-\b{e}_3-(k+1)\b{e}_4}
\Big|_{E_{34}}=
0,
\]    
which implies the result in \eqref{cc1-1} holds true
by the linear independence of Bernstein polynomials on the edge.
Next, using \eqref{cc1-1}, evaluating the expression~\eqref{cc-red2} on the face $F_4$, and recalling that 
$\lambda_4=0$ on this face, we get 
\[
\sum_{
 \overset{\b\alpha\in \mathring{\c I}_{3}^{n+2}}{\alpha_4 = k+1\;\;}}
{\ca}\BB{n-k}{\b{\alpha}-\b{e}_1-(k+1)\b{e}_4}\Big|_{F_4} = 0
\]
which implies the result in \eqref{ca1-1} holds true
by the linear independence of Bernstein polynomials on the face.
This completes the induction proof.

Similar, we can show that $\cb=0$ for all $\b\alpha\in \mathring{\c I}_{3}^{n+2}$, 
hence the set \eqref{BBC-bubble-curl} form a basis for 
$\b H_0(\mathrm{div}^0;T)\cap \mathbb{P}_n^3$, which immediately implies the 
linear independence of functions in the set
\eqref{BBC-bubble}.
This completes the proof of Lemma \ref{lemma:cell-curl-bubble}.
\end{proof}

\bibliographystyle{siam}

\end{document}